\documentclass[11pt]{amsart}

\usepackage[a4paper,margin=1in]{geometry}
\usepackage{amssymb}
\usepackage{mathrsfs}
\usepackage{xcolor}
\usepackage{hyperref}
\usepackage{amsmath,amssymb,graphicx}

\newcommand{\rightovernotleft}{%
  \mathrel{\vcenter{\offinterlineskip
    \halign{\hfil$##$\hfil\cr
      \rightharpoonup\cr
      \noalign{\kern-0.35ex}
      \ooalign{%
        $\leftharpoondown$\cr
        \hidewidth\raisebox{0.2ex}{\scalebox{0.4}{$~\mathbf{/}$}}\hidewidth\cr
      }\cr
    }%
  }}%
}

\newcommand{\rightoverquesmarkleft}{%
  \mathrel{\vcenter{\offinterlineskip
    \halign{\hfil$##$\hfil\cr
      \rightharpoonup\cr
      \noalign{\kern-0.35ex}
      \ooalign{%
        $\leftharpoondown$\cr
        \hidewidth\raisebox{-0.45ex}{\scalebox{0.55}{~?}}\hidewidth\cr
      }\cr
    }%
  }}%
}

\newtheorem{theorem}{Theorem}[section]
\newtheorem{lemma}{Lemma}[section]
\newtheorem{proposition}{Proposition}[section]
\newtheorem{corollary}{Corollary}[section]
\newtheorem{claim}{\quad Claim}
\theoremstyle{definition}

\newtheorem{question}{Question}

\theoremstyle{remark}
\newtheorem{remark}{Remark}

\allowdisplaybreaks[2]

\newcommand{\N}{\mathbb{N}}
\newcommand{\Zp}{\mathbb{N}_{0}}
\newcommand{\Fps}{\mathscr{F}_{\mathrm{ps}}}
\newcommand{\Fs}{\mathscr{F}_{\mathrm{s}}}
\newcommand{\Ft}{\mathscr{F}_{\mathrm{t}}}
\newcommand{\Fpubd}{\mathscr{F}_{\mathrm{pubd}}}
\newcommand{\Finf}{\mathscr{F}_{\mathrm{inf}}}
\newcommand{\PRzero}{\text{-}\mathrm{PR}_{0}}
\newcommand{\orb}{\operatorname{Orb}^{+}}

\newcommand{\Mzero}{{M}_{0}}
\newcommand{\Ezero}{{E}_{0}}
\newcommand{\Lang}{\mathcal{L}}
\newcommand{\Sub}{\operatorname{Sub}}
\newcommand{\Cyl}[1]{[#1]}

\newcommand{\Ascr}{H}

\begin{document}

\title[Disjointness and \(\mathscr{F}\)-PR with respect to zero entropy systems]
{On disjointness and \(\mathscr{F}\)-product recurrence with respect to
zero entropy systems}

\author[X. Wu]{Xinxing Wu}
\address{School of Mathematics and Statistics, Guizhou University of
Finance and Economics, Guiyang, Guizhou 550025, China}
\email{wuxinxing5201314@163.com}


\begin{abstract}
This paper studies disjointness and \(\mathscr{F}\)-product recurrence
with respect to zero-entropy systems. First, we prove that, for a point
whose orbit closure has zero topological entropy,
\[
\text{distality}\rightleftharpoons
\Finf\PRzero\rightleftharpoons \Fpubd\PRzero
\rightleftharpoons \Fps\PRzero \rightovernotleft
\Fs\PRzero.
\]
We also construct a minimal topological dynamical system with uniform
positive entropy which is disjoint from all zero-entropy \(M\)-systems
but is not disjoint from some zero-entropy \(E\)-system, i.e.,
\(\Ezero^{\perp}\subsetneq \Mzero^{\perp}\), giving an affirmative
answer to a question of \cite[W.~Huang, K.~K. Park and X.~Ye,
Topological disjointness from entropy zero systems,
Bull. Soc. Math. France \textbf{135} (2007), 259--282]{HPY}.
Moreover, the same construction shows that there exists an
\(\Fps\PRzero\) point which is not \(\Fpubd\PRzero\), i.e.,
$\Fps\PRzero$ $\nRightarrow \Fpubd\PRzero$,
giving a negative answer to a question
of \cite[P.~Oprocha and G.~H. Zhang,
On weak product recurrence and synchronization of return
times, Adv. Math. \textbf{244} (2013), 395--412]{OZ}.
\end{abstract}

\maketitle

\section{Introduction}

Furstenberg~\cite{Furstenberg} introduced disjointness in ergodic
theory and topological dynamics and used it to characterize classes of
processes and flows by their disjointness relations. In the topological
setting, he proved that every totally transitive system with dense
periodic points is disjoint from all minimal systems and that every
weakly mixing system is disjoint from all minimal distal systems.
He also asked~\cite[Problem~G]{Furstenberg} for
descriptions of the classes of systems
disjoint from all minimal systems and from all distal systems.

A \textit{topological dynamical system} (TDS for short) is a pair
\((X,f)\), where \((X,d)\) is a nontrivial compact metric space and
\(f\colon X\to X\) is a continuous map. According to
Furstenberg~\cite{Furstenberg}, two TDSs \((X,f)\)
and \((Y,g)\) are \textit{disjoint} if their only joining is
\(X\times Y\), where a \textit{joining} is a nonempty closed invariant
subset of \(X\times Y\) projecting onto each coordinate space.
For a class \(\mathscr{X}\) of TDSs, write \(\mathscr{X}^{\perp}\)
for the class of TDSs disjoint from all members of \(\mathscr{X}\).
Let \(\mathscr{D}\), \(\mathscr{M}\), \(M\), and \(E\) denote the
classes of distal systems, minimal systems,
transitive systems with dense minimal points, and transitive systems
admitting an invariant Borel probability measure with full support,
respectively. The zero-entropy subclasses of $\mathscr{M}$, \(M\),
and \(E\) are denoted by $\mathscr{M}_0$, \(\Mzero\), and \(\Ezero\),
respectively.

Petersen~\cite{Petersen} settled the latter question by proving that
$\mathscr{D}^{\perp}$ consists precisely of the weakly mixing minimal
systems. For the former question concerning $\mathscr{M}^{\perp}$,
Huang and Ye~\cite{HY} proved that every TDS in
\(\mathscr{M}^{\perp}\) has dense minimal points, that every transitive
system in \(\mathscr{M}^{\perp}\) is weakly mixing, and that every
weakly mixing system with dense regular minimal points belongs to
\(\mathscr{M}^{\perp}\). They also constructed transitive systems
having no periodic points and a distal system in
\(\mathscr{M}^{\perp}\). Oprocha~\cite{Oprocha2010} stated
that every weakly mixing system with dense
distal points belongs to \(\mathscr M^{\perp}\) (see also~\cite{DSY}).
Dong, Shao and Ye~\cite{DSY} showed that every weakly mixing $M$-system is
disjoint from all minimal PI systems.
Oprocha~\cite{Oprocha2019} gave
a sufficient criterion for a TDS to belong to \(\mathscr{M}^{\perp}\) and
constructed a transitive system $(X, f)$ in
\(\mathscr{M}^{\perp}\) such that the return times set
$N_{f}(x, U)$ is not an \(\mathrm{IP}^{*}\)-set for
some nonempty open set $U$ of $X$ and every $x\in X$. Huang, Shao and
Ye~\cite{HSY} established the converse of Oprocha's sufficient criterion
\cite{Oprocha2019} in a stronger uniform form: a transitive system
is in \(\mathscr{M}^{\perp}\) if and only if it is weakly mixing and
there exists some countable dense subset $D$ of $X$
consisting of minimal points such that for
any minimal system $(Y, g)$, any point $y\in Y$ and any open
neighbourhood $V$ of $y$, and for any nonempty open subset
$U$ of $X$, there exists $x\in D\cap U$ such that
$\{n\in \Zp: f^{n}(x)\in U, g^{n}(y)\in V\}$ is syndetic.
As applications of this characterization, they further proved that finite Cartesian
powers and positive iterates of a transitive system in
\(\mathscr{M}^{\perp}\) remain in \(\mathscr{M}^{\perp}\),
and that a transitive system belongs to \(\mathscr{M}^{\perp}\) if and
only if its hyperspace system does. Then, Huang et al.~\cite{HSXY}
gave an intrinsic characterization for arbitrary systems:
\((X,f)\in\mathscr{M}^{\perp}\) if and only if \(X\) contains minimal
subsystems \((G_n, f|_{G_n})_{n\in \N}\) such that
$\overline{\bigcup_{n\in \N}G_{n}}=X$ and
\((G_n, f|_{G_n})\perp (X, f)\) for every \(n\in\N\). Recently,
Guo et al.~\cite{GQXY} also characterized the transitive systems disjoint
from all totally minimal systems.

For the zero-entropy subclasses,
Blanchard~\cite{Blanchard} proved that every diagonal system
belongs to \(\mathscr{M}_0^\perp\). Huang, Park and Ye~\cite{HPY}
considered the classes \(\Mzero\) and \(\Ezero\). In particular, they proved that every
minimal diagonal system belongs to \(\Mzero^\perp\), while every TDS
in \(\Mzero^\perp\) is minimal and has c.p.e. Moreover, every TDS
in \(\Ezero^\perp\) is minimal and has c.p.e., and a
minimal system belongs to \(\Ezero^\perp\) if each of its invariant
measures defines a measure-theoretic \(K\)-system. They further
characterized the transitive systems in $\mathscr{M}_0^{\perp}$
by introducing the notion of $zm$-sets. Clearly
\(\mathscr{M}_0 \subseteq \Mzero \subseteq \Ezero\)
and \(\Ezero^\perp \subseteq \Mzero^\perp \subseteq \mathscr{M}_0^\perp\).
Huang, Park and Ye~\cite{HPY} further showed that
\(\Mzero^\perp\subsetneq\mathscr{M}_0^\perp\) and
naturally asked whether \(\Ezero^\perp
\subseteq\Mzero^\perp\) is strict.
Dong, Shao and Ye~\cite{DSY} later stated the same
question as open:
\begin{question}[{\cite{HPY,DSY}}]
\label{ques:disjointness}
Is there a TDS in
\(\Mzero^{\perp}\setminus\Ezero^{\perp}\)?
\end{question}

Our first main result gives an affirmative answer
to Question~\ref{ques:disjointness}.

\begin{theorem}
\label{thm:disjointness}
There exists a minimal system in
$\Mzero^{\perp}\setminus\Ezero^{\perp}$
which has u.p.e. In particular,
\(\Ezero^{\perp}\subsetneq\Mzero^{\perp}
\subsetneq\mathscr{M}_0^\perp\).
\end{theorem}

Recurrence is another central issue in the study of topological dynamics.
Regarding ``recurrence in pairs'', which was later called product
recurrence by Auslander and Furstenberg~\cite{AuslanderFurstenberg},
Furstenberg~\cite[Theorem~9.11]{FurstenbergBook}
proved that a point is product recurrent if and only if it is distal
if and only if it is \(\mathrm{IP}^{*}\)-recurrent. Following
Auslander and Furstenberg \cite{AuslanderFurstenberg}, a point $x$ in a TDS is said to be product
recurrent if, for every recurrent point $y$ in any TDS,
$(x, y)$ is recurrent under the product system.
Auslander and Furstenberg~\cite{AuslanderFurstenberg}
extended the equivalence of product recurrence and
distality to more general semigroup actions
and asked whether recurrence in pairs with every minimal point still
forces distality. A point with this property is called weakly
product recurrent~\cite{HaddadOtt}, or equivalently,
\(\Fs\)-product recurrent~\cite{DSY}.
Haddad and Ott~\cite{HaddadOtt} proved that every point with dense orbit in the
one-sided full shift on finitely many symbols, or in a mixing shift of
finite type, is weakly product recurrent. In the full shift such a
point is not distal, giving a negative answer to the question.
Furthermore, Oprocha~\cite{Oprocha2010} gave sufficient
conditions, formulated in terms of weak mixing, under which the set of
weakly product recurrent points is residual.
Glasner and Weiss~\cite{GlasnerWeiss} then constructed a weakly
mixing minimal system every point of which is weakly
product recurrent but not distal.
Combining product recurrence with Furstenberg families, Dong, Shao
and Ye~\cite{DSY} introduced the notion of \(\mathscr{F}\)-PR.
In particular, under their framework, product recurrence is
\(\Finf\)-PR, and weak product recurrence is
\(\Fs\)-PR. They also proved that the orbit closure of every \(\Fs\)-PR point
is an \(M\)-system and every \(\Fps\)-PR point is minimal.
Later, Oprocha and Zhang~\cite{OZ} showed that \(\Fps\)-PR and
\(\Fpubd\)-PR are both equivalent to \(\Finf\)-PR,
and hence to distality, thereby completing the corresponding results
of Dong, Shao and Ye~\cite{DSY}. To sum up, we have
\[
\text{distality}\rightleftharpoons\Finf\text{-PR}
\rightleftharpoons \Fpubd\text{-PR}
\rightleftharpoons \Fps\text{-PR} \rightovernotleft
\Fs\text{-PR}.
\]

Denote by $\mathscr{F}\text{-}\mathrm{PR}_{0}$ the restriction of
$\mathscr{F}$-PR obtained by considering recurrence in pairs
only with points $y$ in TDSs having zero
topological entropy.
Dong, Shao and Ye~\cite{DSY} showed that the orbit closure of every
\(\Fs\text{-}\mathrm{PR}_{0}\) point is an
\(E\)-system, and every \(\Fps\PRzero\) point is minimal.
Since \(\Fs\subseteq \Fps\subseteq \Fpubd
\subseteq \Finf\), one always has
\[
 \Finf\PRzero \Rightarrow\Fpubd\PRzero
 \Rightarrow \Fps\PRzero \Rightarrow
 \Fs\PRzero.
\]
Dong, Shao and Ye~\cite{DSY} established that
$\text{distality}\Leftrightarrow\Finf\PRzero$,
$\Fpubd\PRzero\nRightarrow\Finf\PRzero$, and
$\Fs\PRzero\nRightarrow \Fps\PRzero$.
Therefore
\begin{equation}
\label{eq:PR0-implication-1}
\text{distality}\rightleftharpoons
\Finf\PRzero\rightovernotleft \Fpubd\PRzero
\rightoverquesmarkleft \Fps\PRzero \rightovernotleft
\Fs\PRzero.
\end{equation}
Inspired by these results, Dong, Shao and Ye~\cite{DSY}
posed the following open question (see
also \cite[p.~403]{OZ}):

\begin{question}[{\cite{OZ,DSY}}]
\label{ques:product-recurrence}
Does \(\Fps\PRzero \Rightarrow \Fpubd\PRzero\)
hold?
\end{question}

Our second main result gives a negative answer
to Question~\ref{ques:product-recurrence}.

\begin{theorem}
\label{thm:main}
There exists an \(\Fps\PRzero\) point which is not
\(\Fpubd\PRzero\), i.e., $\Fps\PRzero \nRightarrow
\Fpubd\PRzero$.
\end{theorem}

Combining Theorem~\ref{thm:main} with
\eqref{eq:PR0-implication-1}, we have
\[
\text{distality}\rightleftharpoons
\Finf\PRzero\rightovernotleft \Fpubd\PRzero
\rightovernotleft \Fps\PRzero \rightovernotleft
\Fs\PRzero.
\]

Our third main result shows that, for a point whose orbit closure has
zero topological entropy, the first four properties in the preceding diagram
are equivalent:
\begin{theorem}\label{thm:collapse}
Let \((X, f)\) be a TDS and \(x\in X\). If
\(h_{\mathrm{top}}(f|_{\overline{\orb(x,f)}})=0\),
the following statements are equivalent:
\begin{enumerate}
\renewcommand{\labelenumi}{\textup{(\roman{enumi})}}
\renewcommand{\theenumi}{\roman{enumi}}
\item \(x\) is distal;
\item $x$ is $\Finf\PRzero$;
\item $x$ is $\Fpubd\PRzero$;
\item $x$ is $\Fps\PRzero$.
\end{enumerate}
\end{theorem}

In particular, combining Theorem~\ref{thm:collapse} with
Remark~\ref{Fps0=distai-Remark}, we further obtain the following
relation for points whose orbit closures have zero topological entropy:
\[
\text{distality}\rightleftharpoons
\Finf\PRzero\rightleftharpoons \Fpubd\PRzero
\rightleftharpoons \Fps\PRzero \rightovernotleft
\Fs\PRzero.
\]

\section{Preliminaries}
\label{sec:preliminaries}

This section recalls some basic definitions used throughout the paper.

\subsection{Furstenberg family}

Let \(\N=\{1, 2, 3, \ldots\}\), \(\Zp=\{0, 1, 2, \ldots\}\),
and \(\mathbb{Z}=\{\ldots, -1, 0, 1, \ldots\}\).
For \(x\in\mathbb{R}\), let \(\lfloor x\rfloor\) and
\(\lceil x\rceil\) denote the greatest integer not exceeding
\(x\) and the least integer not less than \(x\), respectively.
Denote by $\mathcal{P}(\Zp)$ the set of all subsets of $\Zp$.
Following Akin~\cite{Akin}, a collection \(\mathscr{F}
\subseteq \mathcal{P}(\Zp)\) is a (\textit{Furstenberg}) \textit{family} if it is
hereditary upward; that is, \(F_{1}\in\mathscr{F}\) and \(F_{1}\subseteq F_{2}\)
imply \(F_{2}\in\mathscr{F}\). A family \(\mathscr{F}\) is
\textit{proper} if it is neither empty nor all of \(\mathcal{P}(\Zp)\).
Since \(\mathscr{F}\) is hereditary upward, it is proper if and only if
\(\varnothing \notin \mathscr{F}\) and \(\Zp \in \mathscr{F}\).
A proper family \(\mathscr{F}\) is a \textit{filter} if it is
closed under intersections, i.e.:
\[
  F_{1},F_{2}\in\mathscr{F}
  \Rightarrow
  F_{1}\cap F_{2}\in\mathscr{F}.
\]

For a family $\mathscr{F}$, the \textit{dual family}
of $\mathscr{F}$, denoted by $\kappa\mathscr{F}$, is defined as
\[
\kappa\mathscr{F}=\{F\in \mathcal{P}(\Zp):
F\cap F_1\neq \varnothing \text{ for any }
F_1\in \mathscr{F}\}.
\]

A set \(A\subseteq \Zp\) is
\begin{itemize}
  \item \textit{thick} if it contains arbitrarily long runs of
  $\Zp$, i.e., for any \(n\in \N\), there exists an \(i\in \Zp\)
  such that \(\{i, i+1, \ldots, i+n\}\subseteq A\);
  \item \textit{syndetic} if it has bounded gaps, i.e.,
  there exists \(L\in\N\) such that \([i, i+L]\cap A
  \neq \varnothing\) for any \(i\in \Zp\);
  \item \textit{piecewise syndetic} if it is an intersection
  of a syndetic set with a thick set.
\end{itemize}

For \(A\subseteq \Zp\), the \textit{upper Banach density} and
\textit{lower Banach density} of \(A\) are, respectively,
\[
 BD^{*}(A)=\limsup_{n-m\to\infty}\frac{\#(A\cap [m, n])}{n-m+1},
 \quad
 BD_{*}(A)=\liminf_{n-m\to\infty}\frac{\#(A\cap [m, n])}{n-m+1}.
\]
%%where \(I\) ranges over intervals of \(\Zp\).
The \textit{upper density} and
\textit{lower density} of \(A\) are, respectively,
\[
 D^{*}(A)=\limsup_{n\to\infty}
 \frac{\#(A\cap[0, n])}{n+1},
 \quad
 D_{*}(A)=\liminf_{n\to\infty}
 \frac{\#(A\cap[0, n])}{n+1}.
\]
We denote by $\Finf$, $\Ft$, $\Fs$, $\Fps$, and $\Fpubd$
the family of all infinite subsets, thick subsets, syndetic
subsets, piecewise syndetic subsets, and subsets with positive upper
Banach density of $\Zp$, respectively.
It is easy to see that (1) $F\in \Fps$ if and only if
$\bigcup_{i=0}^{m} (F-i)\in \Ft$
for some $m\in\Zp$ (\cite[Remark~4.46~(c) and
Theorem~4.49]{HS}), where $F-i=\{n-i: n\in F\}\cap \Zp$;
(2) $\kappa\Ft=\Fs$ and $\kappa\Fs=\Ft$.

\subsection{Topological dynamics}
Let $(X, f)$ be a TDS. A subset $X_1$ of $X$ is
\textit{clopen} if it is both open and closed.
A non-empty closed invariant subset
$X_1\subseteq X$ defines naturally a \textit{subsystem}
$(X_1, f|_{X_1})$  of $(X, f)$. The \textit{orbit} of
a point \(x\in X\) under \(f\) is the set
\(\orb(x,f)=\{f^{n}(x):n\in \Zp\}.\)
The \textit{$\omega$-limit set} of \(x\in X\)
is \(\omega_{f}(x)=\{y\in X: \exists n_{k}\to +\infty
\text{ s.t. } \lim_{k\to\infty}f^{n_{k}}(x)=y\}\).
It is easy to see
\(\omega_{f}(x)=\bigcap_{n\in \Zp}\overline{\{f^{m}(x): m\geq n\}}.\)
For $x\in X$ and $A$, $B\subseteq X$,
we define the \textit{return-time sets}
\(N_{f}(x,B)=\{n \in \N: f^{n}(x)\in B\}\)
and \(N_{f}(A,B)=\{n\in\N:f^{n}(A)\cap B\neq\varnothing\}.\)

Let $(X, f)$ and $(Y, g)$ be two TDSs. If there exists a continuous
surjection $\pi: X\to Y$ such that $\pi \circ f=g\circ \pi$, then
we say that $\pi$ is a \textit{factor map} from $X$ to $Y$, the
system $(Y, g)$ is a \textit{factor} of $(X, f)$ or $(X, f)$ is
an \textit{extension} of $(Y, g)$.

A TDS \((X,f)\) is \textit{transitive} if
\(N_{f}(U,V)\neq\varnothing\) for any non-empty open sets
\(U\), \(V\subseteq X\); \textit{weakly mixing} if
the product system $(X\times X, f\times f)$ is transitive;
and \textit{minimal} if \(\operatorname{Tran}(X,f)=X\).
Equivalently, $(X, f)$ is minimal if and only if it
contains no proper subsystems.

A point \(x\in X\) is called
\begin{itemize}
  \item a \textit{recurrent point} of \((X,f)\) if
\(N_{f}(x,U)\neq\varnothing\) for any neighborhood \(U\) of \(x\).
  \item a \textit{transitive point} of \((X,f)\) if \(\overline{\orb(x,f)}=X\).
  \item a \textit{minimal point} if
  \((\overline{\orb(x,f)},f|_{\overline{\orb(x,f)}})\) is minimal.
\item a \textit{distal point} if $\liminf_{n\to \infty}d(f^{n}(x), f^{n}(y))
  >0$ for any $y\in \overline{\orb(x,f)}\setminus \{x\}$.
\end{itemize}
Denote by
\(\operatorname{Rec}(X,f)\) and \(\operatorname{Tran}(X,f)\), respectively,
the sets of all recurrent points and transitive points of \((X,f)\).
It is well known that $x\in X$ is a minimal point of
$(X, f)$ if and only if for any neighborhood $U$ of $x$,
$N_{f}(x, U)\in \Fs$.

We say that a TDS \((X, f)\) is
\begin{itemize}
  \item an \textit{\(M\)-system} if it is transitive and the set of all
  minimal points is dense in \(X\);
  \item an \textit{\(E\)-system} if it is transitive and has an
  invariant Borel probability measure with full support;
  \item an \textit{$M_0$-system} if it is an $M$-system with zero topological
  entropy;
  \item an \textit{$E_0$-system} if it is an $E$-system with zero topological
  entropy.
\end{itemize}
We denote by $\mathscr{M}$, $M$, $E$, $M_0$, and $E_0$ the
classes of all minimal systems, $M$-systems,
\(E\)-systems, $M_0$-systems, and $E_0$-systems,
respectively. The following return-time characterizations for $M$-
and $E$-systems are standard; see \cite[Lemma~2.1~\textup{(1)}]{HY}
for part~\textup{(\ref{item:ME-ps})} and \cite[Lemma~3.6]{HPY}
for part~\textup{(\ref{item:ME-pubd})}.

\begin{lemma}[{\cite[Lemma~2.1\textup{(1)}]{HY}; \cite[Lemma~3.6]{HPY}}]
\label{lem:ME}
Let $(X, f)$ be a transitive TDS and \(x\in \operatorname{Tran}(X,f)\).
Then,
\begin{enumerate}
\renewcommand{\labelenumi}{\textup{(\roman{enumi})}}
\renewcommand{\theenumi}{\roman{enumi}}
\item\label{item:ME-ps} \(x\) is \(\Fps\)-recurrent if and only if \((X, f)\) is an
      \(M\)-system.
\item\label{item:ME-pubd} \(x\) is \(\Fpubd\)-recurrent if and only if \((X, f)\) is an
      \(E\)-system.
\end{enumerate}
\end{lemma}

Let $(X, f)$ and $(Y, g)$ be two TDSs. For \(A\subseteq X\times Y\),
let \(\mathrm{Proj}_{i}(A)\) denote the projection of \(A\) to the
\(i\)-th coordinate, \(i=1,2\). We say that a non-empty closed
set \(J\subseteq X\times Y\) is a \textit{joining} of \((X,f)\) and \((Y,g)\)
if it is invariant under the joint action of \(f\times g\) and
projects onto each coordinate space, i.e.,
$\mathrm{Proj}_{1}(J)=X$ and $\mathrm{Proj}_{2}(J)=Y$.
If each joining is equal to \(X\times Y\), then we say that \((X,f)\) and
\((Y,g)\) are \textit{disjoint} and denote this fact by
\((X,f) \perp (Y,g)\) or simply by \(f\perp g\).
Clearly, if $(X, f)\perp (Y, g)$, then $f$ and $g$
are surjective maps.

If $\mathcal{U}$ is an open cover of $X$, let $N(\mathcal{U})$
denote the minimum cardinality of a subcover of $\mathcal{U}$.
Given two open covers $\mathcal{U}$ and $\mathcal{V}$ of $X$,
$\mathcal{U}$ is said to be a \textit{refinement} of $\mathcal{V}$,
written $\mathcal{U}\succcurlyeq \mathcal{V}$, if
for any $U\in \mathcal{U}$, there exists $V\in \mathcal{V}$
such that $U\subseteq V$; their \textit{join}
$\mathcal{U}\vee\mathcal{V}$
is defined by
\[
\mathcal{U}\vee\mathcal{V}
=\{U\cap V:U\in\mathcal{U},\ V\in\mathcal{V}\}.
\]
Clearly $\mathcal{U}\vee\mathcal{V}\succcurlyeq \mathcal{U}$
and $\mathcal{U}\vee\mathcal{V}\succcurlyeq \mathcal{V}$.
Similarly we can define the join $\bigvee_{i=1}^{n}\mathcal{U}_i$
of any finite collection of open covers of $X$.
The \textit{topological entropy of $\mathcal{U}$
with respect to $f$} is
\[
h_{\mathrm{top}}(f,\mathcal{U})
=\lim_{n\to\infty}\frac{1}{n}\log N\left(
\bigvee_{i=0}^{n-1}f^{-i}\mathcal{U}\right),
\]
and the \textit{topological entropy of $f$} is
\[
h_{\mathrm{top}}(f)
=\sup_{\mathcal{U}}h_{\mathrm{top}}(f,\mathcal{U}),
\]
where $\mathcal{U}$ ranges over all open covers of $X$.
By \cite{Walters}, we know that if $\mathcal{U}\succcurlyeq
\mathcal{V}$, then $h_{\mathrm{top}}(f,\mathcal{U})\geq
h_{\mathrm{top}}(f,\mathcal{V})$.

A pair \((x_{1},x_{2})\in X\times X\) is an \textit{entropy pair}
if \(x_{1}\neq x_{2}\) and for any disjoint closed neighborhoods
\(U_{i}\) of \(x_{i}\), the open cover
\(\{X\setminus U_{1}, X\setminus U_{2}\}\)
has positive entropy. Let \(E_{2}(X,f)\) denote the set of
all entropy pairs for $(X, f)$.

According to Blanchard~\cite{Blanchard}, a TDS \((X,f)\)
\begin{itemize}
  \item has \textit{uniform positive entropy}
  (u.p.e.) if \(E_{2}(X, f)=X\times X\setminus\{(x,x):x\in X\}\).
  \item is a \textit{diagonal system} if \(E_{2}(X, f)\supseteq
  \{(x, f(x)): x\in X,\ x\neq f(x)\}\).
\end{itemize}
Clearly u.p.e.\ implies diagonality \cite{HPY}.
Furthermore, Blanchard \cite{Blanchard}
proved that a diagonal system is disjoint from all
minimal systems with zero entropy and that
$h_{\mathrm{top}}(f)>0$ if and only if $E_{2}(X, f)
\neq \varnothing$. For minimal
diagonal systems, Huang, Park and Ye \cite{{HPY}}
obtained the following result:
\begin{lemma}[{\cite[Theorem~2.7]{HPY}}]
\label{thm:HPY}
If \((X,f)\) is a minimal diagonal TDS, then
\((X, f)\in\Mzero^{\perp}.\)
\end{lemma}

\subsection{Product recurrence}

Let $(X, f)$ be a TDS and $\mathscr{F}$ be a family. Following
Auslander and Furstenberg \cite{AuslanderFurstenberg}
and Dong, Shao and Ye~\cite{DSY}, we say that a point
\(x\in X\) is
\begin{itemize}
  \item \textit{\(\mathscr{F}\)-recurrent} if \(N_{f}(x,U)\in\mathscr{F}\)
for any neighborhood \(U\) of \(x\).
  \item \textit{product recurrent} if for any recurrent point \(y\) in
any TDS \((Y,g)\), the pair \((x, y)\) is recurrent
for \((X\times Y, f\times g)\).
\item \textit{weakly product recurrent} if for any recurrent point \(y\)
in any minimal TDS \((Y,g)\), the pair \((x, y)\) is recurrent
for \((X\times Y, f\times g)\).
  \item \textit{\(\mathscr{F}\)-product recurrent}
(\(\mathscr{F}\)-PR for short) if for any \(\mathscr{F}\)-recurrent
point \(y\) in any TDS \((Y,g)\), the pair \((x, y)\) is recurrent
for \((X\times Y, f\times g)\).
  \item \textit{\(\mathscr{F}\text{-}\mathrm{PR}_{0}\)} if for
any $\mathscr{F}$-recurrent point $y$ in any TDS $(Y, g)$ having zero
topological entropy, the pair $(x, y)$ is recurrent
for $(X\times Y, f\times g)$.
\end{itemize}

It is noteworthy that although the definition of \
\(\mathscr{F}\text{-}\mathrm{PR}_{0}\) differs from the
original one in \cite[Definition~5.1]{DSY}, under which only
$\overline{\orb(y, g)}$ is required to have zero topological entropy,
it is easy to see that the two definitions for
$\mathscr{F}\text{-}\mathrm{PR}_{0}$ are equivalent.
Meanwhile, by the definitions,
it is easy to see that a point is product
recurrent if and only if it is $\Finf$-PR, and is
weakly product recurrent if and only if it is $\Fs$-PR.

\subsection{Symbolic dynamics}
For \(\Sigma=\{0, 1\}\),
endow \(\Sigma_2:=\Sigma^{\Zp}\) with the product
topology of the discrete topology on \(\Sigma\).
The \textit{shift map} \(\sigma\colon\Sigma_2 \to \Sigma_2\)
is defined by \((\sigma(x))_{i}=x_{i+1}\). A \textit{subshift}
is the restriction \(\sigma|_{X_1}\) of \(\sigma\) to
a nonempty closed subset \(X_1\subseteq \Sigma_2\) satisfying
\(\sigma(X_1)\subseteq X_1\). For \(m\in \N\) and a word
\(w=w_0w_1\cdots w_{m-1}\in \Sigma^m\),
define the \textit{cylinder set}
\[
  \Cyl{w}=\{x\in \Sigma_2:
  x_i=w_i\text{ for }0\leq i\leq m-1\}.
\]
For a subshift \(X_1\subseteq \Sigma_2\), put
\(\Cyl{w}_{X_1}=X_1\cap\Cyl{w}\) and let \(\Lang_{n}(X_1)\) and
\(\Lang(X_1)\) denote its sets of words of length \(n\) and
of all finite words in $X_1$, respectively,
i.e., $\Lang_{n}(X_1)=\{x_jx_{j+1}\cdots x_{j+n-1}: j\in \Zp, x
=(x_i)_{i\in \Zp} \in X_1\}$ and $\Lang(X_1)=\bigcup_{n\in \N}\Lang_{n}(X_1)$.
It is easy to see that the cylinder sets \(\Cyl{w}_{X_1}\),
\(w\in \Lang(X_1)\), are clopen subsets
of $X_1$ and form a base for \(X_1\).

For \(w\in\Lang(\Sigma_{2})\), write \(|w|\) for its length and
\(\Sub(w)\) for the set of all words occurring in \(w\),
i.e., $\Sub(w)=\{w_i\cdots w_j: 0\leq i\leq j\leq |w|-1\}$.

For any $x=(x_i)_{i\in \Zp}\in \Sigma_2$ and $n\in \N$, let
\(p_{x}(n)\) denote the number of words of length
\(n\) occurring in \(x\), i.e.,
$p_{x}(n)=\#(\{x_{j}x_{j+1}\cdots x_{j+n-1}:
j\in \Zp\})$. It is easy to see that
$p_{x}(n)=\Lang_{n}(\overline{\orb(x, \sigma)})$.
According to \cite[Theorem~7.13~(i)]{Walters},
\begin{equation}\label{eq:entropy-sigma}
h_{\mathrm{top}}(\sigma|_{\overline{\orb(x,\sigma)}})
=\lim_{n\to \infty}\frac{1}{n}\log p_{x}(n).
\end{equation}

\section{$\Fps\PRzero$ on zero-entropy orbit closures}

This section first provides a sufficient condition for
\(\Fps\PRzero\) and establishes an entropy-preserving synchronization
construction. We then use this construction to characterize
\(\Fps\PRzero\) points whose orbit closures have zero topological entropy,
thereby proving Theorem~\ref{thm:collapse} and the strictness
of \(\Fps\PRzero \rightovernotleft
\Fs\PRzero\) stated in
Remark~\ref{Fps0=distai-Remark}.

Combining Lemmas~\ref{lem:ME} and \ref{thm:HPY}, we
first have the following sufficient condition for
\(\Fps\PRzero\):
\begin{proposition}
\label{prop:product-recurrence}
Let \((X,f)\) be a minimal diagonal system.
Then every point of \(X\) is \(\Fps\PRzero\).
\end{proposition}

\begin{proof}
Fix any \(x\in X\), and let \(y\) be an \(\Fps\)-recurrent point of
a zero-entropy TDS \((Y,g)\).
Lemma~\ref{lem:ME}~\textup{(\ref{item:ME-ps})} gives
\((\overline{\orb(y,g)}, g|_{\overline{\orb(y,g)}})\in \Mzero\).
Set \(J=\omega_{f\times g}(x,y).\)
Clearly \(J\) is a nonempty closed invariant set
of $f\times g|_{\overline{\orb(y,g)}}$. Meanwhile, by
the minimality of $(X, f)$, there exists an increasing sequence
$(n_{j})$ such that $\lim_{j\to \infty}f^{n_{j}}(x)=x$.
By the compactness of $\overline{\orb(y, g)}$,
there exists a subsequence $(n'_{j})$ such that
$(g^{n'_{j}}(y))$ has a limit $y_{0}\in \overline{\orb(y, g)}$.
Then
\[
\lim_{j\to \infty}(f\times g)^{n'_{j}}
(x, y)=(x, y_{0})\in J,
\]
implying
$x\in \mathrm{Proj}_{1}(J)$. Thus
$\orb(x, f)\subseteq \mathrm{Proj}_{1}
(\orb((x, y_{0}), f\times g))\subseteq
\mathrm{Proj}_{1}(J)$, and hence
$X=\overline{\orb(x, f)}=
\mathrm{Proj}_{1}(J)$.
Similarly, it can be verified
$\overline{\orb(y,g)}=\mathrm{Proj}_{2}(J)$.
Therefore $J$ is a joining of $(X, f)$
and $(\overline{\orb(y,g)},
g|_{\overline{\orb(y,g)}})$. This, together with
Lemma~\ref{thm:HPY}, implies
\[
 \omega_{f\times g}(x,y)=J=X\times \overline{\orb(y,g)}.
\]
In particular \((x, y)\) is a recurrent point
of $f\times g$. This means that $x$ is \(\Fps\PRzero\).
\end{proof}

We next establish the entropy-preserving synchronization
construction needed below.

\begin{lemma}
\label{prop:entropy-sync}
Let \(a\in\Sigma_2\) be a minimal point with \(a_0=1\). Then, for
any \(T\in \Ft\), there exists \(z\in\Sigma_2\)
with \(z_0=1\) such that
\begin{enumerate}
\renewcommand{\labelenumi}{\textup{(\roman{enumi})}}
\renewcommand{\theenumi}{\roman{enumi}}
\item\label{item:sync-recurrence}
\(z\) is \(\Fps\)-recurrent;
\item\label{item:sync-support}
\(N_\sigma(z, \Cyl{1})\subseteq
N_\sigma(a, \Cyl{1})\cap T\);
\item\label{item:sync-entropy}
\(h_{\mathrm{top}}(\sigma|_{\overline{\orb(z,\sigma)}})
 =h_{\mathrm{top}}(\sigma|_{\overline{\orb(a,\sigma)}}).\)
\end{enumerate}
\end{lemma}

\begin{proof}
By \(T\in \Ft\), there exist intervals
\(I_{n}=[b_{n}, b_{n}+n-1]\cap \N \subseteq T\),
\(n \in \N\), such that
\begin{equation}\label{eq:island-separation}
 b_{n}>n \text{ and }
 b_{n+1}-(b_{n}+n)>n
 \quad (\forall n \in \N).
\end{equation}

Noting that every \(n\in \N\) has a unique representation \(n=2^{m} (2j-1)\),
where \(m \in \Zp\) and \(j\in \N\), we shall use this representation throughout.
For \(m, j\in \N\), enumerate the elements of
$[b_{2^{m}(2j-1)},
 b_{2^{m}(2j-1)}+2^{m}(2j-1)-m-1]
 \cap N_\sigma(a,\Cyl{a_{0}\cdots a_{m}})$
increasingly as
\[
 \gamma_{1}^{(m,j)}<\cdots<\gamma_{N(m,j)}^{(m,j)},
\]
where
\(N(m,j)=\#([b_{2^{m}(2j-1)},
 b_{2^{m}(2j-1)}+2^{m}(2j-1)-m-1]
 \cap N_\sigma(a,\Cyl{a_{0}\cdots a_{m}}))\),
and set
\[
\mathscr{G}_{m,j}
 =\bigcup_{\ell=1}^{\left\lfloor\frac{N(m,j)}{m(m+1)}\right\rfloor}
 [\gamma_{\ell m(m+1)}^{(m,j)},\gamma_{\ell m(m+1)}^{(m,j)}+m],
 \quad
 \mathscr{G}_{m}=\bigcup_{j\in\N}\mathscr{G}_{m,j},
\]
and
\[
\mathscr{G}_{m}^{L}
=\bigcup_{j\in\N}
 \left\{\gamma_{\ell m(m+1)}^{(m,j)}:
 1\leq\ell\leq\left\lfloor\frac{N(m,j)}{m(m+1)}\right\rfloor\right\}.
 \]
By these constructions, the uniqueness of the representation
\(2^{m}(2j-1)\), and \eqref{eq:island-separation},
it can be verified that
\begin{itemize}
  \item \(\mathscr{G}_{m}=\bigcup_{k\in\mathscr{G}_{m}^{L}}[k,k+m]\);
  \item The family
\(\{[k,k+m]:m\in\N,\ k\in\mathscr{G}_{m}^{L}\}
\cup\{I_{2j-1}:j\in\N\}\cup\{\{0\}\}\) is pairwise disjoint.
\end{itemize}

Now we define \(z=(z_{i})_{i\in\Zp}\in\Sigma_{2}\) as follows:
\begin{itemize}
  \item \(z_{0}=1\);
  \item \(z_{i}=a_{i}\) for
\(i\in\bigcup_{j\in\N}I_{2j-1}\);
  \item \(z_{i}=0\) for
\(i\in\Zp\setminus(\{0\}\cup\bigcup_{j\in\N}I_{2j-1}
 \cup\bigcup_{m\in \N}\mathscr{G}_{m})\);
  \item The remaining values of \(z\) are defined inductively
on the sets of coordinates \(\mathscr{G}_{m}\), \(m\in\N\). The pairwise
disjointness above guarantees that the domains of all these assignments are
pairwise disjoint. First, we observe that, for \(m\in\N\), by
\eqref{eq:island-separation} and
the construction of \(\mathscr{G}_{m',j}\),
it follows that, for any \(m'\geq m\) and any \(j\in\N\),
\[
 \mathscr{G}_{m',j}\subseteq I_{2^{m'}(2j-1)}
 \subseteq[b_{2^{m'}(2j-1)},+\infty)
 \subseteq [2^{m},+\infty)
 \subseteq[m+1,+\infty),
\]
and hence
\begin{equation}\label{eq:[0,m]-values}
 [0,m]\cap\bigcup_{m'\geq m}\mathscr{G}_{m'}=\varnothing.
\end{equation}

When \(m=1\), this shows that \(z_{0}\) and \(z_{1}\) are assigned by the
first three items. In particular, \(z_{0}=1\) and \(z_{1}=0\).
For any \(k\in\mathscr{G}_{1}^{L}\), take \(z_{k}z_{k+1}=z_{0}z_{1}=10\).
The pairwise disjointness above ensures that the values on \(\mathscr{G}_{1}\)
are uniquely determined.

Now let \(m\geq2\), and suppose that the values on \(\mathscr{G}_{m'}\)
have been defined for all \(1\leq m'<m\). By \eqref{eq:[0,m]-values},
each value $z_{i}$, $0\leq i\leq m$, has either been assigned initially
by the first three items or determined on some \(\mathscr{G}_{m'}\) with
\(1\leq m'<m\), so \(z_{0}\cdots z_{m}\) is determined.
Then, for any \(k\in\mathscr{G}_{m}^{L}\), take
\(z_{k}\cdots z_{k+m}=z_{0}\cdots z_{m}\). The
pairwise disjointness above ensures that the values on \(\mathscr{G}_{m}\)
are uniquely determined.
\end{itemize}

Thus these assignments define \(z\) uniquely. In particular,
it can be verified
\begin{equation}\label{eq:z-definition}
 z_{i}=
 \begin{cases}
  1, & i=0,\\
  a_{i}, & i\in \bigcup_{j\in \N}I_{2j-1},\\
  z_{i-k}, & i\in[k,k+m] \text{ for } k\in \mathscr{G}_{m}^{L} \ (m\in \N),\\
  0, & \text{otherwise}.
 \end{cases}
\end{equation}

Next we verify the three conditions:

\smallskip

(i) For any neighborhood $U$ of $z$, there exists some $m\in \N$
such that $\Cyl{z_{0}\cdots z_{m}}\subseteq U$. Since \(a\) is a
minimal point, \(N_\sigma(a, \Cyl{a_{0}\cdots a_{m}})\) is syndetic,
implying that there exists \(L_{m}\in\N\) such that for any $i\in \Zp$,
\([i, i+L_{m}-1]\cap N_\sigma(a,\Cyl{a_{0}\cdots a_{m}})\neq\varnothing\).
This, together with the definition of $N(m, j)$ and
\(\lim_{j\to \infty}(2^{m}(2j-1)-m)=+\infty\), yields that there
exists \(j_{m}\in\N\) such that, for any \(j\geq j_{m}\),
$N(m,j)\geq 2m(m+1)$. Meanwhile, by the choices of
$\gamma_{1}^{(m,j)}, \gamma_{2}^{(m,j)},\ldots,
\gamma_{N(m,j)}^{(m,j)}$, we have
that, for any \(j\geq j_{m}\),
\[
\begin{cases}\gamma_{1}^{(m,j)}-b_{2^{m}(2j-1)}\leq L_{m}-1,\\
\gamma_{\ell+1}^{(m,j)}-\gamma_{\ell}^{(m,j)}\leq L_{m}\quad
(1\leq \ell<N(m,j)),\\
b_{2^{m}(2j-1)}+2^{m}(2j-1)-m-1
-\gamma_{N(m,j)}^{(m,j)}\leq L_{m}-1,
\end{cases}
\]
implying
\[
\begin{cases}
0\leq \gamma_{m(m+1)}^{(m,j)}-b_{2^{m}(2j-1)}
\leq m(m+1)L_{m},\\
\gamma_{(\ell+1)m(m+1)}^{(m,j)}-\gamma_{\ell m(m+1)}^{(m,j)}
\leq m(m+1)L_{m} \quad (1\leq \ell <
\left\lfloor\frac{N(m,j)}{m(m+1)}\right\rfloor), \\
b_{2^{m}(2j-1)}+2^{m}(2j-1)-m-1-
\gamma_{\left\lfloor\frac{N(m,j)}{m(m+1)}\right\rfloor m(m+1)}^{(m,j)}
\leq m(m+1)L_{m}.
\end{cases}
\]
Together with \(\left\{\gamma_{\ell m(m+1)}^{(m,j)}:1\leq\ell\leq
\left\lfloor\frac{N(m,j)}{m(m+1)}\right\rfloor\right\}\subseteq\mathscr{G}_{m}^{L}\), we obtain
\[
 \bigcup_{j\geq j_{m}}[{b_{2^{m}(2j-1)}},
 {b_{2^{m}(2j-1)}+2^{m}(2j-1)-1}-m(m+1)L_{m}-m]
 \subseteq\bigcup_{i=0}^{m(m+1)L_{m}+m}({\mathscr{G}_{m}^{L}}-i),
\]
i.e., $\bigcup_{i=0}^{m(m+1)L_{m}+m}({\mathscr{G}_{m}^{L}}-i)
\in \Ft$. Thus \(\mathscr{G}_{m}^{L}\in\Fps\).

On the other hand, for any \(n\in\mathscr{G}_{m}^{L}\),
the definition of \(z\) gives
\(z_{n}\cdots z_{n+m}=z_{0}\cdots z_{m}\), and hence
\(\mathscr{G}_{m}^{L}\subseteq
 N_\sigma(z, \Cyl{z_{0}\cdots z_{m}})
 \subseteq N_\sigma(z, U)\in\Fps\) by \(\mathscr{G}_{m}^{L}\in\Fps\).
 Therefore \(z\) is \(\Fps\)-recurrent.

\smallskip

(ii)
We first inductively prove that \(z_{n}\leq a_{n}\) for all $n\in \Zp$.

For $n=0$, it is clear that $z_{0}= a_{0}$.
For $n\in \Zp$, suppose that $z_{i}\leq a_{i}$ holds for
all $0\leq i\leq n$. It suffices to check $z_{n+1}\leq a_{n+1}$.
We consider the following two cases:

1) If $n+1\in \Zp\setminus (\bigcup_{m\in\N}
\bigcup_{k\in\mathscr{G}_{m}^{L}}[k, k+m])$,
by the definition of $z$, it is clear that
$z_{n+1}\leq a_{n+1}$.

2) If $n+1\in \bigcup_{m\in\N}
\bigcup_{k\in\mathscr{G}_{m}^{L}}[k, k+m]$, then
there exist $m\in \N$ and $k\in \mathscr{G}_{m}^{L}$
such that $n+1=k+i$ for some $0\leq i\leq m$.
Since $i<k+i=n+1$, the induction hypothesis,
together with the definition of $z$ and
$\mathscr{G}_{m}^{L}\subseteq N_\sigma(a,\Cyl{a_{0}\cdots a_{m}})$, gives
$z_{n+1}=z_{k+i}=z_{i}\leq a_{i}=a_{k+i}=a_{n+1}$.

Thus $z_{n}\leq a_{n}$ for all $n\in\Zp$.

Now, for any \(n\in N_\sigma(z,\Cyl{1})\),
we have \(z_{n}=1\),
and then \(a_{n}=1\). By \eqref{eq:z-definition}, either
\(n\in I_{2j-1}\) for some \(j\in\N\), or
\(n\in\mathscr{G}_{m,j}\subseteq I_{2^{m}(2j-1)}\) for some
\(m,j\in\N\), implying \(n\in T\), and thus
\( n\in N_\sigma(a,\Cyl{1})\cap T.\) Therefore
\(N_\sigma(z, \Cyl{1})\subseteq N_\sigma(a,\Cyl{1})\cap T\).

\smallskip

(iii) For simplicity, denote
\(X_{a}=\overline{\orb(a,\sigma)}\) and
\(X_{z}=\overline{\orb(z,\sigma)}\).

\smallskip

iii-1) $h_{\mathrm{top}}(\sigma|_{X_{z}})
\geq h_{\mathrm{top}}(\sigma|_{X_{a}})$.

\smallskip

Fix any \(n\in \N\). For any $i\in \Zp$, since $\sigma^{i}(a)$
is a minimal point in $X_a$, we have
\(N_{\sigma}(\sigma^{i}(a),
[a_i\cdots$ $a_{i+n-1}]_{X_{a}})\in \Fs\),
implying \(N_{\sigma}(\sigma^{i}(a),
\Cyl{a_i\cdots a_{i+n-1}}_{X_{a}})+i\subseteq
N_{\sigma}(a,\Cyl{a_i\cdots a_{i+n-1}}_{X_{a}})\in\Fs.\)
Moreover,  since $\{a_i\cdots a_{i+n-1}: i\in \Zp\}$
is a finite set, there exists
\(L(n)\in\N\) such that for any $j\in \Zp$,
$[j, j+L(n)-1]\cap
N_{\sigma}(a,\Cyl{u}_{X_{a}})\neq \varnothing$ for all $u\in
\{a_{i}\cdots a_{i+n-1}:
i\in \Zp\}$. This means that every
length-(\(L(n)+n-1\)) word occurring in $a$ contains
all length-\(n\) words occurring in \(a\).

Let \(j_0=L(n)+n+1\). Noting that \(z_{b_{2j_0-1}}\cdots
z_{b_{2j_0-1}+2j_0-2}=a_{b_{2j_0-1}}\cdots
a_{b_{2j_0-1}+2j_0-2}\) by \eqref{eq:z-definition}
and $|a_{b_{2j_0-1}}\cdots a_{b_{2j_0-1}+2j_0-2}|=2j_0-1>L(n)+n-1$,
the word \(z_{b_{2j_0-1}}\cdots z_{b_{2j_0-1}+L(n)+n-2}\) is a
length-(\(L(n)+n-1\)) word occurring in \(a\), and
therefore contains all length-\(n\) words occurring in
\(a\). Thus \(\{a_i\cdots a_{i+n-1}: i\in \Zp\}
 \subseteq \{z_i\cdots z_{i+n-1}: i\in \Zp\},\)
so \(p_{z}(n)\geq p_{a}(n)\). Hence
\begin{align*}
h_{\mathrm{top}}
(\sigma|_{X_{z}})
= & \lim_{n\to \infty}\frac{1}{n}\log \#\{z_i\cdots z_{i+n-1}: i\in \Zp\}\\
\geq & \lim_{n\to \infty}\frac{1}{n}\log \#\{a_i\cdots a_{i+n-1}: i\in \Zp\}\\
= & h_{\mathrm{top}}(\sigma|_{X_{a}}).
\end{align*}

iii-2) $h_{\mathrm{top}}(\sigma|_{X_{z}})
\leq h_{\mathrm{top}}(\sigma|_{X_{a}})$.

\smallskip

Fix any \(n\in\N\).
By \eqref{eq:island-separation}, every interval $[t, t+n-1]$,
$t\in \Zp$, that avoids \(I_1,\ldots,I_n\) meets at most one interval among
\(I_{n+1},I_{n+2},\ldots\), since, for every \(s\geq n+1\),
\(b_{s+1}-(b_s+s-1)>s+1>n.\)
Moreover, by \eqref{eq:z-definition},
\begin{equation}\label{eq:z-support}
 \{i\in\Zp:z_i=1\}\subseteq\{0\}\cup\bigcup_{s\in\N}I_s.
\end{equation}

For each \(t\in\Zp\) such that \(z_t\cdots z_{t+n-1}\neq 0^n\),
we consider the following four cases:

\textup{(1)} Suppose that \([t,t+n-1]\cap
 \left(\{0\}\cup\bigcup_{s=1}^{n}I_s\right)\neq\varnothing\).
Noting that if \([t,t+n-1]\cap I_s\neq\varnothing\), then
\(t\in[b_s-n+1,b_s+s-1]\cap\Zp,\)
whereas \([t,t+n-1]\cap\{0\}\neq\varnothing\) implies \(t=0\).
Consequently, the number of length-$n$ words in this case
is at most
\[
 1+\sum_{s=1}^{n}(s+n-1)\leq 2n^2.
\]

\textup{(2)} Suppose that \([t,t+n-1]\cap \left(\{0\}
\cup\bigcup_{s=1}^{n}I_s\right)=\varnothing\) and
\([t, t+n-1]\cap (\bigcup_{j\in \N}I_{2j-1})\neq \varnothing\).
By the preceding observation, \([t,t+n-1]\) meets exactly one interval
among \(I_{n+1},I_{n+2},\ldots\), and this interval is odd-indexed.
Hence, by \eqref{eq:z-definition}, there exist
\(u,\alpha,\beta\in\Zp\) and \(1\leq d\leq n\) with
$\alpha+d+\beta=n$ such that
\[
z_{t}\cdots z_{t+n-1}=0^\alpha a_u\cdots a_{u+d-1}0^\beta.
\]
For any fixed \(1\leq d\leq n\), the equation \(\alpha+d+\beta=n\) has \(n-d+1\)
solutions in \((\alpha,\beta)\in\Zp^2\). Hence, the number
of length-$n$ words in this case is at most
\[
 \sum_{d=1}^{n}(n-d+1)p_a(d)\leq n^2p_a(n).
\]

\textup{(3)} Suppose that \([t,t+n-1]\cap \left(\{0\}\cup
\bigcup_{s=1}^{n}I_s\right)=\varnothing\) and
\([t, t+n-1]\cap \left(\bigcup_{j\in \N}
\bigcup_{m=1}^{\lceil\sqrt{n}\rceil}
I_{2^m(2j-1)}\right)\neq
\varnothing\). By the preceding observation, \([t,t+n-1]\) meets exactly
one interval among \(I_{n+1},I_{n+2},\ldots\), and this interval is
\(I_{2^m(2j-1)}\) for unique \(m,j\in\N\) with
\(1\leq m\leq\lceil\sqrt{n}\rceil\).

For each fixed \(1\leq m\leq\lceil\sqrt{n}\rceil\), let
\[
\begin{split}
 \mathscr{S}_{3,m}=\bigg\{z_t\cdots z_{t+n-1}:{}&t\in\Zp,\
 z_t\cdots z_{t+n-1}\neq0^n,\
 [t,t+n-1]\cap\bigg(\{0\}\cup\bigcup_{s=1}^{n}I_s\bigg)=\varnothing,\\
 & [t,t+n-1]\cap \bigg(\bigcup_{j\in \N}I_{2^m(2j-1)}
 \bigg)\neq\varnothing
 \bigg\}.
\end{split}
\]

For any \(w\in\mathscr{S}_{3,m}\), define
\[
\begin{split}
 t(w)=\min\bigg\{t\in\Zp:{}&
 z_t\cdots z_{t+n-1}=w,\
 [t,t+n-1]\cap\bigg(\{0\}\cup\bigcup_{s=1}^{n}I_s\bigg)
 =\varnothing,\\
 & [t,t+n-1]\cap \bigg(\bigcup_{j\in \N}I_{2^m(2j-1)}
 \bigg)\neq\varnothing\bigg\}.
\end{split}
\]
Then, there exists a unique \(j(w)\in\N\) such that
\([t(w),t(w)+n-1]\) meets \(I_{2^m(2j(w)-1)}\). Put
\(s(w)=2^m(2j(w)-1),\)
\(c(w)=\max\{b_{s(w)},t(w)-m\},\) and
\(e(w)=\min\{b_{s(w)}+s(w)-m-1,t(w)+n-1\}.\)

We first note that, for any \(k\in\mathscr{G}_{m}^{L}\cap I_{s(w)}\),
\begin{equation}\label{eq:equivalence}
 [k,k+m]\cap[t(w),t(w)+n-1]\neq\varnothing
 \Longleftrightarrow k\in[c(w),e(w)].
\end{equation}
In fact, by the construction of \(\mathscr{G}_{m}^{L}\) and
\(k\in\mathscr{G}_{m}^{L}\cap I_{s(w)}\), we have
\(b_{s(w)}\leq k\leq b_{s(w)}+s(w)-m-1.\)
Together with
\[
 [k,k+m]\cap[t(w),t(w)+n-1]\neq\varnothing
 \Longleftrightarrow
 t(w)-m\leq k\leq t(w)+n-1,
\]
this implies that $[k,k+m]\cap[t(w),t(w)+n-1]
\neq\varnothing$
if and only if
\[
 \max\{b_{s(w)},t(w)-m\}\leq k
 \leq\min\{b_{s(w)}+s(w)-m-1,t(w)+n-1\},
\]
which is equivalent to \(k\in[c(w),e(w)]\).

Next, since \(z_{t(w)}\cdots z_{t(w)+n-1}=w\neq0^n\),
there exists \(t(w)\leq i\leq t(w)+n-1\) such that
\(z_i=1\). This, together with \eqref{eq:z-support} and
the definition of \(t(w)\), implies
\(i\in I_{s(w)}\). Hence, by \eqref{eq:z-definition}, there exists
\(k\in\mathscr{G}_{m}^{L}\cap I_{s(w)}\) such that
\(i\in[k,k+m].\)
Since also \(i\in[t(w),t(w)+n-1]\),
\([k,k+m]\cap[t(w),t(w)+n-1]\neq\varnothing.\)
The equivalence above therefore gives \(k\in[c(w),e(w)]\),
and thus
\([c(w),e(w)]\cap\mathscr{G}_{m}^{L}\neq\varnothing.\)
Hence the integer
\[
 \nu(w)=\min\left\{r\in\{1,\ldots,N(m,j(w))\}:
 \gamma_r^{(m,j(w))}\in[c(w),e(w)]\right\}
\]
is well defined. Set
\[
 v(w)=a_{c(w)}\cdots a_{c(w)+n+2m-1},\quad
 \delta(w)=c(w)-t(w),
\]
\[
\epsilon(w)=e(w)-t(w), \quad
 \rho(w)=\nu(w)\bmod m(m+1).
\]
Clearly \(c(w)\leq e(w)\) by
\([c(w),e(w)]\cap\mathscr{G}_{m}^{L}\neq\varnothing\). Moreover
\(t(w)-m\leq c(w)\leq e(w)\leq t(w)+n-1.\)
Consequently
\[
\begin{cases}
 -m\leq\delta(w)\leq
 \epsilon(w)\leq n-1,\\
 0\leq e(w)-c(w)\leq n+m-1,\\
 e(w)+m\leq c(w)+n+2m-1.
\end{cases}
\]

Then, define
\[
 \Phi_m:\mathscr{S}_{3,m}\rightarrow \mathscr{P}_m,
 \quad
 \Phi_m(w)=(v(w),\rho(w),\delta(w),\epsilon(w)),
\]
where
\[
\mathscr{P}_m=
\left\{(v,\rho,\delta,\epsilon):
\begin{array}{l}
v\in\{a_r\cdots a_{r+n+2m-1}:r\in\Zp\},\\
\rho\in\{0,\ldots,m(m+1)-1\},\\
(\delta,\epsilon)\in\mathbb Z^2,\
-m\leq\delta\leq\epsilon\leq n-1
\end{array}
\right\}.
\]

We claim that $\Phi_m$ is injective.

For any \(w,w'\in\mathscr{S}_{3,m}\) with \(w\neq w'\),
by $w=z_{t(w)}\cdots z_{t(w)+n-1}$ and
$w'=z_{t(w')}\cdots z_{t(w')+n-1}$,
there exists \(0\leq h\leq n-1\) such that
\(z_{t(w)+h}\neq z_{t(w')+h}.\)
Without loss of generality, assume
\(z_{t(w)+h}=1\) and \(z_{t(w')+h}=0.\)

If
\((v(w),\delta(w),\epsilon(w))\neq
 (v(w'),\delta(w'),\epsilon(w')),\)
then \(\Phi_m(w)\neq\Phi_m(w')\).
Now suppose that
\((v(w),\delta(w),\epsilon(w))=
 (v(w'),\delta(w'),\epsilon(w')).\)
For each \(x\in\{w,w'\}\), the definitions of \(c(x)\) and \(e(x)\) give
\([c(x),e(x)]\subseteq [b_{s(x)},b_{s(x)}+s(x)-m-1],\)
and
\(e(x)-c(x)=\epsilon(x)-\delta(x).\) Hence, for any
\(\gamma\in\Zp\), the definition of
\(\gamma_{1}^{(m,j(x))},\ldots,
\gamma_{N(m,j(x))}^{(m,j(x))}\) gives
\[
\begin{split}
 &\gamma\in
 \{\gamma_r^{(m,j(x))}:1\leq r\leq N(m,j(x))\}
 \cap[c(x),e(x)]\\
 &\quad\Longleftrightarrow
 \gamma\in[c(x),e(x)]\cap N_\sigma(a,\Cyl{a_0\cdots a_m})\\
 &\quad\Longleftrightarrow
 \gamma=c(x)+d\text{ for some }
 d\in[0,\epsilon(x)-\delta(x)]\cap\Zp
 \text{ such that }\\
 &\hspace{48mm}
 v(x)_d\cdots v(x)_{d+m}=a_0\cdots a_m,
\end{split}
\]
implying
\begin{equation}\label{eq:local-return-coordinates}
\begin{split}
 &\{\gamma_r^{(m,j(x))}:1\leq r\leq N(m,j(x))\}
 \cap[c(x),e(x)]\\
 &\quad =\{c(x)+d:0\leq d\leq\epsilon(x)-\delta(x),\
 v(x)_d\cdots v(x)_{d+m}=a_0\cdots a_m\}.
\end{split}
\end{equation}

By $(v(w),\delta(w),\epsilon(w))=
 (v(w'),\delta(w'),\epsilon(w'))$,
we have
$\{d\in[0, \epsilon(w)-\delta(w)]\cap \Zp:
 v(w)_d\cdots$ $v(w)_{d+m}=a_0\cdots a_m\}
 =\{d\in[0, \epsilon(w')-\delta(w')] \cap \Zp:
 v(w')_d\cdots v(w')_{d+m}=a_0\cdots a_m\}$,
and enumerate it as \(\{d_1<\cdots<d_R\}.\)
This, together with the definition of \(\nu(x)\)
and the increasing order of
\(\gamma_1^{(m,j(x))},\ldots,\gamma_{N(m,j(x))}^{(m,j(x))}\),
implies that, for \(x\in\{w,w'\}\) and \(1\leq r\leq R\),
\begin{equation}\label{eq:gamma-index-shift}
 c(x)+d_r=\gamma_{\nu(x)+r-1}^{(m,j(x))}
\end{equation}
and by \(\rho(x)\equiv\nu(x)\pmod{m(m+1)}\),
\begin{equation}\label{eq:selected-gamma-index}
 c(x)+d_r\in\mathscr{G}_m^L
 \Longleftrightarrow
 \nu(x)+r-1\equiv 0\!\!\!\pmod{m(m+1)}
 \Longleftrightarrow
 \rho(x)+r-1 \equiv 0\!\!\!\pmod{m(m+1)}.
\end{equation}

Since \(z_{t(w)+h}=1\),
\eqref{eq:z-support} and the fact that
\([t(w),t(w)+n-1]\) meets only \(I_{s(w)}\) give
\(t(w)+h\in I_{s(w)}\). Hence, by \eqref{eq:z-definition}, there exists
\(k\in\mathscr{G}_m^L\cap I_{s(w)}\) such that
\(t(w)+h\in[k,k+m]\) and \(z_{t(w)+h}=z_{t(w)+h-k}=1.\)
This implies \(t(w)+h\in [k,k+m]\cap[t(w),t(w)+n-1]\neq\varnothing\),
and thus \(k\in[c(w),e(w)]\) by \eqref{eq:equivalence}.
Consequently,
\(k\in\mathscr{G}_m^L\cap[c(w),e(w)]\). By the construction of
\(\mathscr{G}_m^L\),
\[
 k\in
 \{\gamma_r^{(m,j(w))}:1\leq r\leq N(m,j(w))\}
 \cap[c(w),e(w)].
\]
From \eqref{eq:gamma-index-shift}, there exists a unique
\(1\leq r_0\leq R\) such that \(k=c(w)+d_{r_0}\). Since
\(k\in\mathscr{G}_m^L\), applying \eqref{eq:selected-gamma-index} to this
\(r_0\) yields \(\rho(w)+r_0-1\equiv0\pmod{m(m+1)}.\)

Moreover, \(k=c(w)+d_{r_0}\) and
\(\delta(w)=c(w)-t(w)\), together with the equality
\(z_{t(w)+h}=z_{t(w)+h-k}=1\) above, give
\[
 z_{t(w)+h}
 =z_{h-\delta(w)-d_{r_0}}=1.
\]
If
\(\rho(w')+r_0-1\equiv0\pmod{m(m+1)}\), then
\eqref{eq:selected-gamma-index} gives
\(k'=c(w')+d_{r_0}\in\mathscr{G}_m^L\). Moreover,
\begin{equation}\label{eq:relative-coordinate}
 t(w')+h-k'
 =h-\delta(w')-d_{r_0}
 =h-\delta(w)-d_{r_0}
 =t(w)+h-k.
\end{equation}
Since \(t(w)+h\in[k,k+m]\), this gives
\(t(w')+h\in[k',k'+m]\). Together
with~\eqref{eq:relative-coordinate} and
\eqref{eq:z-definition}, we have
\[
z_{t(w')+h}= z_{t(w')+h-k'}=z_{h-\delta(w')-d_{r_0}}
 =z_{h-\delta(w)-d_{r_0}}=1,
\]
contrary to \(z_{t(w')+h}=0\). Therefore
\(\rho(w')+r_0-1\not\equiv0\pmod{m(m+1)},\)
and hence \(\rho(w)\neq\rho(w')\). Thus
\(\Phi_m(w)\neq\Phi_m(w')\), proving that \(\Phi_m\) is injective.

Applying this claim yields
\[
 \#\mathscr{S}_{3,m}\leq \#\mathscr{P}_m
 \leq m(m+1)\binom{n+m+1}{2}
 p_a(n+2m).
\]
Finally, let
\(\mathscr{S}_3=\bigcup_{m=1}^{\lceil\sqrt{n}\rceil}\mathscr{S}_{3,m}.\)
Hence the number of length-$n$
words in this case is at most $\# \mathscr{S}_3$, and
\[
\begin{aligned}
 \#\mathscr{S}_3
 &\leq\sum_{m=1}^{\lceil\sqrt{n}\rceil}\#\mathscr{S}_{3,m}
 \leq\sum_{m=1}^{\lceil\sqrt{n}\rceil}
 m(m+1)\binom{n+m+1}{2}p_a(n+2m)\\
 &\leq
 2\sqrt{n}\,(2\sqrt{n})^{2}
 \frac{(3n)^{2}}{2}
 p_a(n+2\lceil\sqrt{n}\rceil)
 =36n^{7/2}p_a(n+2\lceil\sqrt{n}\rceil).
\end{aligned}
\]

\textup{(4)} Suppose that \([t,t+n-1]\cap
\left(\{0\}\cup\bigcup_{s=1}^{n}I_s\right)=\varnothing\) and
\([t,t+n-1]\cap \left(\bigcup_{j\in\N}
\bigcup_{m=\lceil\sqrt{n}\rceil+1}^{\infty}
I_{2^m(2j-1)}\right)$ $\neq\varnothing\).
By the preceding observation, \([t,t+n-1]\) meets
exactly one interval among \(I_{n+1},I_{n+2},\ldots\), and this interval
is \(I_{2^m(2j-1)}\) for unique \(m,j\in\N\) with
\(m\geq\lceil\sqrt{n}\rceil+1\).

We first prove by strong induction on
\(m\geq\lceil\sqrt{n}\rceil+1\) that, for
every \(1\leq\ell\leq n\), every nonzero word
\(z_t\cdots z_{t+\ell-1}\)
satisfying
\([t,t+\ell-1]\cap
\left(\{0\}\cup\bigcup_{s=1}^{n}I_s\right)=\varnothing\)
and \([t,t+\ell-1]\cap
\left(\bigcup_{j\in\N}I_{2^m(2j-1)}\right)\neq\varnothing\)
can be written as
\[
 z_t\cdots z_{t+\ell-1}=0^\alpha u0^\beta,
 \quad \alpha,\beta\in\Zp,\quad 1\leq d\leq\ell,\quad
 \alpha+d+\beta=\ell,
\]
where \(u=z_q\cdots z_{q+d-1}{\neq0^d}\)
and \(q\in\Zp\) satisfies one of the following:

(a) $[q,q+d-1]\cap
\left(\{0\}\cup \bigcup_{s=1}^{n}I_s\right)
\neq\varnothing;$

(b) $[q,q+d-1]\cap \left(\{0\}\cup
\bigcup_{s=1}^{n}I_s\right)
=\varnothing$ and
\([q,q+d-1]\cap\left(\bigcup_{j'\in\N}I_{2j'-1}\right)
\neq\varnothing;\)

(c) $[q,q+d-1]\cap
\left(\{0\}\cup \bigcup_{s=1}^{n}I_s\right)
=\varnothing$ and
\([q,q+d-1]\cap
\left(
\bigcup_{j'\in\N}
\bigcup_{m'=1}^{\lceil\sqrt{n}\rceil}
I_{2^{m'}(2j'-1)}
\right)
\neq\varnothing.\)

\smallskip

For
\(m=\lceil\sqrt{n}\rceil+1\), let \(1\leq\ell\leq n\) and
\(z_t\cdots z_{t+\ell-1}\) be a nonzero word satisfying
the two conditions above. The
preceding observation gives a unique \(j\in\N\) such that
\([t,t+\ell-1]\cap I_{2^m(2j-1)}\neq\varnothing.\)
Since $z_t\cdots z_{t+\ell-1}\neq 0^{\ell}$, there exists
\(t\leq i\leq t+\ell-1\) such that \(z_i=1\). This, together with
\eqref{eq:z-definition} and \eqref{eq:z-support}, implies that
there exists \(k\in\mathscr{G}_{m}^{L}\cap I_{2^m(2j-1)}\)
such that \(i\in[k,k+m]\). Moreover, consecutive
components of \(\mathscr{G}_{m,j}\) have distance at least
\(m(m+1)-m=m^2>n\geq\ell.\)
Hence \([t,t+\ell-1]\) meets a unique component \([k,k+m]\). Put
\[
 q_0=\max\{0,t-k\},\quad
 d_0=\min\{m,t+\ell-1-k\}-q_0+1,
\]
\[
 \alpha_0=\max\{0,k-t\},\quad
 \beta_0=\ell-\alpha_0-d_0.
\]
Then
\(z_{t}\cdots z_{t+\ell-1}
 =0^{\alpha_0}z_{q_0}\cdots z_{q_0+d_0-1}0^{\beta_0}\)
and \([q_0,q_0+d_0-1]\subseteq[0, m].\)
Since \(i\in[k,k+m]\) and \(z_i=z_{i-k}=1\),
the word $z_{q_0}\cdots z_{q_0+d_0-1}$ is nonzero.
If the word $z_{q_0}\cdots z_{q_0+d_0-1}$ satisfies one of
\textup{(a)}--\textup{(c)}, this case follows.
Otherwise, by \(z_{q_0}\cdots z_{q_0+d_0-1}\neq 0^{d_0}\)
and \eqref{eq:z-support}, there exist
\(m'\geq \lceil\sqrt{n}\rceil+1\) and \(j'\in\N\)
such that
\([q_0,q_0+d_0-1]\cap I_{2^{m'}(2j'-1)}\neq\varnothing.\)
As in the preceding discussion, there exists
\(k'\in\mathscr{G}_{m'}^{L}\cap I_{2^{m'}(2j'-1)}\) such that
\([q_0,q_0+d_0-1]\cap[k',k'+m']\neq\varnothing\), implying
\([q_0,q_0+d_0-1]\cap\mathscr{G}_{m',j'}\neq\varnothing\).
Noting that
\([q_0,q_0+d_0-1]\subseteq[0,m]\) and
\(\mathscr{G}_{m',j'}\subseteq\mathscr{G}_{m'}\), we have
\([0,m]\cap\mathscr{G}_{m'}\neq\varnothing\). This, together with
\eqref{eq:[0,m]-values}, implies
\[
 \lceil\sqrt{n}\rceil+1 \leq m'<m=\lceil\sqrt{n}\rceil+1,
\]
which is impossible. Thus it holds for \(m=\lceil\sqrt{n}\rceil+1\).

\smallskip

Now let \(m>\lceil\sqrt{n}\rceil+1\), and suppose that the assertion
holds for all \(\lceil\sqrt{n}\rceil+1\leq m'<m.\)
Let \(1\leq\ell\leq n\) and
\(z_t\cdots z_{t+\ell-1}\) be a nonzero word satisfying
the two conditions above. Arguing as in the case of
\(m=\lceil\sqrt{n}\rceil+1\), we obtain
\[
 z_t\cdots z_{t+\ell-1}
 =0^{\alpha_0}z_{q_0}\cdots z_{q_0+d_0-1}0^{\beta_0},
 \quad \alpha_0+d_0+\beta_0=\ell,
\]
where \([q_0,q_0+d_0-1]\subseteq[0,m]\) and
\(z_{q_0}\cdots z_{q_0+d_0-1}\) is nonzero. If the word
\(z_{q_0}\cdots z_{q_0+d_0-1}\) satisfies one of
\textup{(a)}--\textup{(c)}, the assertion follows. Otherwise, the same
argument gives \(m'\geq\lceil\sqrt{n}\rceil+1\), \(j'\in\N\), and
\(k'\in\mathscr{G}_{m'}^{L}\cap I_{2^{m'}(2j'-1)}\) such that
\([q_0,q_0+d_0-1]\cap[k',k'+m']\neq\varnothing\), implying
\([q_0,q_0+d_0-1]\cap\mathscr{G}_{m',j'}\neq\varnothing\). Noting that
\([q_0,q_0+d_0-1]\subseteq[0,m]\) and
\(\mathscr{G}_{m',j'}\subseteq\mathscr{G}_{m'}\), we have
\([0,m]\cap\mathscr{G}_{m'}\neq\varnothing\). This, together with
\eqref{eq:[0,m]-values}, implies
\[
 \lceil\sqrt{n}\rceil+1\leq m'<m.
\]

Since condition \textup{(a)} does not hold for
\(z_{q_0}\cdots z_{q_0+d_0-1}\), we get
\([q_0,q_0+d_0-1]\cap
 \left(\{0\}\cup\bigcup_{s=1}^{n}I_s\right)=\varnothing.\)
Moreover, by \([q_0,q_0+d_0-1]\cap\mathscr{G}_{m',j'}\neq\varnothing\)
and \(\mathscr{G}_{m',j'}\subseteq I_{2^{m'}(2j'-1)}
 \subseteq\bigcup_{j\in\N}I_{2^{m'}(2j-1)},\)
we obtain
\([q_0,q_0+d_0-1]\cap
 \left(\bigcup_{j\in\N}I_{2^{m'}(2j-1)}\right)\neq\varnothing.\)
Then the induction hypothesis for $m'$ gives
\[
 z_{q_0}\cdots z_{q_0+d_0-1}=0^{\alpha_1}u0^{\beta_1},
 \quad \alpha_1+d+\beta_1=d_0,
\]
where \(u\) satisfies one of \textup{(a)}--\textup{(c)}. Thus
\[
 z_t\cdots z_{t+\ell-1}
 =0^{\alpha_0+\alpha_1}u0^{\beta_0+\beta_1},
 \quad (\alpha_0+\alpha_1)+d+
 (\beta_0+\beta_1)=\ell,
\]
which completes the strong induction.

Fix any \(1\leq\ell\leq n\). If a word \(u=z_q\cdots
z_{q+\ell-1}\) satisfies \textup{(a)}, then either \(q=0\), or
\([q,q+\ell-1]\cap I_s\neq\varnothing\) for some \(1\leq s\leq n\).
In the latter case, \({q\in[b_s-\ell+1,b_s+s-1]\cap\Zp.}\)
Therefore, the number of words \(u\) satisfying
\textup{(a)} is at most
\[
1+\sum_{s=1}^{n}(s+\ell-1)\leq 2n^2.
\]

Since \(\ell\leq n\), condition \textup{(b)} implies
\([q,q+\ell-1]\cap
 \left(\{0\}\cup\bigcup_{s=1}^{\ell}I_s\right)=\varnothing.\)
Therefore, applying the argument in case \textup{(2)}
with the word length (\(n\)) replaced by \(\ell\), the number of words
\(u\) satisfying \textup{(b)} is at most
\[
 {\sum_{d=1}^{\ell}(\ell-d+1)p_a(d)
 \leq \ell^2p_a(\ell)\leq n^2p_a(n).}
\]

For words \(u\) satisfying \textup{(c)}, fix
\(1\leq h\leq\lceil\sqrt{n}\rceil\). The injectivity argument in case
\textup{(3)} for each fixed \(m\) does not use the relation
between \(m\) and the word length. Hence, with the word
length \(n\) and the parameter \(m\)
replaced by \(\ell\) and \(h\), respectively, it gives the
bound
\[
 h(h+1)\binom{\ell+h+1}{2}p_a(\ell+2h).
\]
Summing over
\(1\leq h\leq\lceil\sqrt{n}\rceil\) as in case \textup{(3)}, the number
of words \(u\) satisfying \textup{(c)} is at most
\[
 36n^{7/2}
 p_a\bigl(n+2\lceil\sqrt{n}\rceil\bigr).
\]

By the monotonicity of \(p_a\) and \(p_a(n)\geq1\), each
of the preceding three bounds is at most
\[
 36n^{7/2}
 p_a\bigl(n+2\lceil\sqrt{n}\rceil\bigr).
\]
For a fixed word \(u\) of length \(\ell\), the
equation \(\alpha+\ell+\beta=n\) has \(n-\ell+1\) solutions in
\((\alpha,\beta)\in\Zp^2\). Therefore, the number of words in case
\textup{(4)} is at most
\[
 3\sum_{\ell=1}^{n}(n-\ell+1)
 36n^{7/2}p_a\bigl(n+2\lceil\sqrt{n}\rceil\bigr)
 \leq 108n^{11/2}
 p_a(n+2\lceil\sqrt{n}\rceil).
\]

For \(n\geq2\), after including the word \(0^n\) in the estimate for case
\textup{(1)}, each of the four resulting bounds is at
most the bound obtained in case \textup{(4)}. Hence,
\begin{equation*}
 p_z(n)\leq 4\cdot108n^{11/2}
 p_a (n+2\lceil\sqrt{n}\rceil)
 =432n^{11/2}p_a (n+2\lceil\sqrt{n}\rceil).
\end{equation*}
Therefore, by \eqref{eq:entropy-sigma},
\begin{align*}
 h_{\mathrm{top}}(\sigma|_{X_z})
 &=\lim_{n\to\infty}\frac{1}{n}\log p_z(n)
 \leq \limsup_{n\to\infty}\frac{1}{n}\log
 (432n^{11/2}p_a (n+2\lceil\sqrt{n}\rceil))\\
 &=\limsup_{n\to\infty}
 \left(\frac{\log(432n^{11/2})}{n}
 +\frac{n+2\lceil\sqrt{n}\rceil}{n}
 \frac{\log p_a(n+2\lceil\sqrt{n}\rceil)}
 {n+2\lceil\sqrt{n}\rceil}\right)\\
 &=h_{\mathrm{top}}(\sigma|_{X_a}).
\end{align*}

Combining the two inequalities proves
assertion~\textup{(\ref{item:sync-entropy})} and completes the proof.
\end{proof}

We now encode return times in a minimal zero-entropy
system by a minimal binary point.

For a surjective TDS \((X, f)\), recall its \emph{natural
extension} \((\tilde{X},\tilde{f})\) defined as follows:
\[
 \tilde{X}
 =\left\{(x_1,x_2,\cdots)\in\prod_{n\in \N} X:
 f(x_{n+1})=x_n,\ n\in\N\right\},
\]
and
\[
 \tilde{f}(x_1,x_2,\cdots)=(f(x_1),x_1,x_2,\cdots).
\]
The map $\tilde{f}$ is a homeomorphism, and the projection
$p_1: \tilde{X}\to X$ sending $(x_1, x_2, \cdots)$
to $x_1$ is a factor map.
It is well known that $(X, f)$ and $(\tilde{X},\tilde{f})$
share many dynamical properties. In particular,
$h_{\mathrm{top}}(f)=h_{\mathrm{top}}(\tilde{f})$
(\cite[Corollary~3.2]{YeInverseLimitEntropy}),
and $(X, f)$ is minimal if and only if
$(\tilde{X}, \tilde{f})$ is minimal
(\cite[Section~2.2]{HuangKolyadaZhang}).
\begin{lemma}%%[Zero-entropy binary realization]
\label{lem:binary-realization}
Let \((X,f)\) be a minimal zero-entropy system and $x\in X$.
Then, for any neighborhood $U$ of $x$, there exists a minimal
point \(a\in\Sigma_2\) such that
\begin{enumerate}
\renewcommand{\labelenumi}{\textup{(\roman{enumi})}}
\renewcommand{\theenumi}{\roman{enumi}}
\item\label{item:a0=1}
 $a_0=1$;
\item\label{item:a-support}
\(N_\sigma(a, \Cyl{1}) \subseteq N_f(x,U)\);
\item\label{item:a-entropy}
\(h_{\mathrm{top}}(\sigma|_{\overline{\orb(a,\sigma)}})=0\).
\end{enumerate}
\end{lemma}

\begin{proof}
Let \((\tilde{X}, \tilde{f})\) be the natural extension of \((X,f)\).
Then \((\tilde{X},\tilde{f})\) is a minimal zero-entropy system.
By
\cite[Proposition~2.4]{GlasnerWeissQF}, there exist a zero-dimensional
space \(Z\), a homeomorphism \(g\colon Z\to Z\), and a continuous
surjection \(\pi\colon Z\to\tilde{X}\)
such that
\[
 \tilde{f}\circ\pi=\pi\circ g
 \text{ and }
 h_{\mathrm{top}}(g)
 =h_{\mathrm{top}}(\tilde{f})=0.
\]
Fix a minimal point \(z^{*}\in Z\) and put
\(Z^{*}=\overline{\orb(z^{*},g)}\).  Since \(\pi(Z^{*})\) is a nonempty closed
invariant subset of the minimal system \((\tilde{X},\tilde{f})\),
we have \(\pi(Z^{*})=\tilde{X}\), and thus $\pi|_{Z^{*}}
: Z^{*}\to \tilde{X}$
is a factor map. Take
\(z_0\in (p_1\circ\pi|_{Z^*})^{-1}(x)\).
Since \(Z^{*}\) is zero-dimensional and $U$ is
a neighborhood of $x$, there exists a clopen neighborhood
\(V\) of \(z_0\) in \(Z^*\) such that
\(V\subseteq (p_1\circ\pi|_{Z^*})^{-1}(U)\).
Then it can be verified that the map
\(\Phi\colon Z^{*}\to\Sigma_2\) defined by
\begin{equation}\label{eq:Phi-Def}
(\Phi(z))_i=1 \text{ if and only if } g^i(z)\in V
\quad (i \in \Zp,\ z\in Z^{*})
\end{equation}
is continuous and satisfies
\(\Phi\circ(g|_{Z^{*}})=\sigma\circ\Phi\), i.e.,
\(\Phi: Z^{*}\to \Phi(Z^{*})\) is a factor map.
Noting that
\((Z^{*},g|_{Z^{*}})\) is minimal, we have
that \((\Phi(Z^{*}),
\sigma|_{\Phi(Z^{*})})\) is minimal, and thus
\(a=\Phi(z_0)\) is a minimal point of $\sigma$.
Clearly $a_0=(\Phi(z_0))_0=1$ by $z_0\in V$
and \eqref{eq:Phi-Def}.
Moreover
\(
h_{\mathrm{top}}(\sigma|_{\overline{\orb(a,\sigma)}})
\leq h_{\mathrm{top}}(g|_{Z^{*}})
\leq h_{\mathrm{top}}(g)=0.
\)
Meanwhile, for any $n\in N_\sigma(a, \Cyl{1})$,
i.e., \(a_n=1\),
we have $g^{n}(z_0)\in V$ by \eqref{eq:Phi-Def},
and thus
\[
f^n(x)=(f^{n}\circ p_1\circ \pi|_{Z^*}) (z_0)
=(p_1\circ \pi|_{Z^*}\circ g^n)(z_0)
\in (p_1\circ \pi|_{Z^*})(V)\subseteq U,
\]
implying \(N_\sigma(a, \Cyl{1})\subseteq N_f(x,U)\).
\end{proof}

Combining the preceding two lemmas, we obtain a zero-entropy
subshift containing a transitive \(\Fps\)-recurrent point whose returns to a
suitable neighborhood are simultaneously constrained by a return-time set of
the original system and a prescribed thick set.

\begin{theorem}%%[Zero-entropy synchronization]
\label{thm:zero-sync}
Let \((X, f)\) be a minimal zero-entropy system and \(x\in X\).
Then, for any neighborhood \(U\) of \(x\) and any \(T\in \Ft\),
there exist a zero-entropy subshift
\((Z,\sigma|_Z)\) of $\Sigma_2$, a transitive \(\Fps\)-recurrent point
\(z\in Z\), and a neighborhood \(W\) of \(z\) such that
\[
 N_{\sigma|_Z}(z, W)\subseteq N_f(x, U)\cap T.
\]
\end{theorem}

\begin{proof}
Let \(a\) be given by Lemma~\ref{lem:binary-realization} applied to $x$
and $U$, and let \(z\) be given by Lemma~\ref{prop:entropy-sync}
applied to \(a\) and \(T\). Then take \(Z=\overline{\orb(z,\sigma)}\)
and \(W=\Cyl{1}_Z\). Clearly \(W\) is a neighborhood of \(z\) in \(Z\)
by \(z_0=1\) and \(z\) is a transitive point in \((Z, \sigma|_Z)\).
Moreover, Lemma~\ref{prop:entropy-sync}~\textup{(\ref{item:sync-recurrence})}
implies that \(z\) is \(\Fps\)-recurrent in \((Z, \sigma|_Z)\). Furthermore,
Lemma~\ref{lem:binary-realization}~\textup{(\ref{item:a-entropy})} and
Lemma~\ref{prop:entropy-sync}~\textup{(\ref{item:sync-entropy})} give
\(h_{\mathrm{top}}(\sigma|_Z)=0.\) Finally
Lemma~\ref{prop:entropy-sync}~\textup{(\ref{item:sync-support})}
and Lemma~\ref{lem:binary-realization}~\textup{(\ref{item:a-support})}
yield
\(N_{\sigma|_Z}(z, W)\subseteq N_\sigma(z, \Cyl{1})
\subseteq N_\sigma(a, \Cyl{1})\cap T
\subseteq N_f(x, U)\cap T.\)
\end{proof}

In view of \eqref{eq:PR0-implication-1},
Theorem~\ref{thm:collapse} reduces to the following lemma.

\begin{lemma}\label{thm:Fps0=distal}
Let \((X,f)\) be a TDS and \(x\in X\). If
\(h_{\mathrm{top}}(f|_{\overline{\orb(x,f)}})=0\),
then \(x\) is $\Fps\PRzero$ if and only if
$x$ is distal.
\end{lemma}

\begin{proof}
$(\Longleftarrow)$. It follows directly from \eqref{eq:PR0-implication-1}.

\smallskip

$(\Longrightarrow)$. Let $X_{x}=\overline{\orb(x,f)}$.
Since \(x\) is \(\Fps\PRzero\),
by \cite[Theorem~5.6]{DSY} and
\(h_{\mathrm{top}}(f|_{X_{x}})=0\),
$f|_{X_{x}}$
is a minimal zero-entropy system.
Suppose, on the contrary, that \(x\) is not distal.
Then there exists \(y\in X_{x}\setminus\{x\}\) such that $\liminf_{n\to \infty}
d(f^{n}(x), f^{n}(y))=0$. Choose open neighborhoods \(V\) of \(x\)
and \(U\) of \(y\) with $\overline{U}\cap \overline{V}=\varnothing$,
and choose
\(\varepsilon=\frac{1}{2}\min\{d(x_1, x_2):
x_1\in\overline{V}, x_2\in\overline{U}\}>0\).
By $\liminf_{n\to \infty}
d(f^{n}(x), f^{n}(y))=0$, the set
$T=\{n\in \N: d(f^{n}(x), f^{n}(y))
<\varepsilon\}$ is thick.

Applying Theorem~\ref{thm:zero-sync} to the minimal zero-entropy system
\((X_{x}, f|_{X_{x}})\), the point \(y\), the neighborhood \(U\cap X_{x}\) of
$y$, and the thick set \(T\), we obtain a zero-entropy subshift
\((Z,\sigma|_Z)\) of $\Sigma_2$, a transitive
\(\Fps\)-recurrent point \(z\in Z\), and a neighborhood \(W\) of \(z\)
such that \(N_{\sigma|_Z}(z,W)\subseteq N_{f|_{X_{x}}}(y,U\cap X_{x})\cap T
\subseteq N_{f}(y, U)\cap T.\)
This implies that, for any \(n\in N_{\sigma|_Z}(z, W)\),
\(f^{n}(y)\in U\) and \(d(f^{n}(x), f^{n}(y))<\varepsilon\). Together with
the choice of $\varepsilon$, we get \(f^{n}(x)\notin V\). Consequently
\(
 N_{f\times(\sigma|_Z)}((x,z), V\times W)=\varnothing,
\)
so \((x,z)\) is not recurrent, contradicting
\(x\in \Fps \PRzero\). Therefore \(x\) is distal.
\end{proof}

\begin{proof}[Proof of Theorem~\ref{thm:collapse}]
It follows directly from Lemma~\ref{thm:Fps0=distal} and
\eqref{eq:PR0-implication-1}.
\end{proof}

\begin{remark}\label{Fps0=distai-Remark}
Choose a weakly mixing doubly minimal system
\((X, f)\) having zero topological entropy (\cite{Weiss}).
Then, for any minimal system \((Y, g)\), every pair in
\(X\times Y\) is recurrent (\cite[Theorem~0.4]{GlasnerWeiss2015}).
This implies that every point in $X$ is \(\Fs\PRzero\).
On the other hand, noting that every weakly mixing minimal system
has no distal points
(\cite[Theorem~9.12]{FurstenbergBook}),
Lemma~\ref{thm:Fps0=distal} thus implies
that every point of \(X\) is not \(\Fps\PRzero\). Therefore,
even among points in minimal zero-entropy systems, it holds that
\(\Fps\PRzero \rightovernotleft
\Fs\PRzero.\)
\end{remark}

\section[Proofs of Theorems~\ref{thm:disjointness} and \ref{thm:main}]
{Proofs of Theorems~\ref{thm:disjointness}
and \ref{thm:main}}

Theorem~\ref{thm:collapse} implies that the orbit closure
of every point in \(\Fps\PRzero\setminus\Fpubd\PRzero\) has positive
topological entropy. This explains why positive topological entropy
is indispensable for the strict separation in Theorem~\ref{thm:main}.

This section constructs two subshifts $(Y, \sigma|_{Y})$
and $(X, \sigma|_{X})$ satisfying the following conditions:
\begin{enumerate}
\renewcommand{\labelenumi}{\textup{(\roman{enumi})}}
  \item $(Y, \sigma|_{Y})\in E_{0}\setminus M_{0}$;
  \item $(X, \sigma|_{X})$ is minimal and has u.p.e.;
  \item There exist $a\in \operatorname{Tran}(Y, \sigma|_{Y})$
  and $\hat{x}\in X$, and
   neighborhoods $\Cyl{1}_{Y}$ and $\Cyl{1}_{X}$ of
   $a$ and $\hat{x}$, respectively, such that
   $N_{\sigma|_{Y}}(a, \Cyl{1}_{Y})\notin \Fps$
   and
   $N_{\sigma|_{Y}}(a, \Cyl{1}_{Y})\cap
   N_{\sigma|_{X}}(\hat{x}, \Cyl{1}_{X})=\varnothing$;
  \item \((X,\sigma|_{X})\not\perp(Y,\sigma|_{Y}).\)
\end{enumerate}
The final subsection combines these constructions to prove
Theorems~\ref{thm:disjointness} and \ref{thm:main}, establishing that
$(X, \sigma|_{X})\in \Mzero^{\perp}
\setminus \Ezero^{\perp}$ and
\({\hat{x}}\in\Fps\PRzero\setminus
\Fpubd\PRzero\).

\smallskip

Let
\begin{equation}
\label{eq:q-definition}
 q_{0}=1,\ q_{1}=4, \text{ and } q_{k+1}=2q_{k}^{2}
 \ (k\in \N).
\end{equation}
It can be verified that:
\begin{enumerate}
\renewcommand{\theenumi}{Q-\arabic{enumi}}
\renewcommand{\labelenumi}{\textup{(\theenumi)}}
\item\label{item:q-divisibility}
\(q_{i} | \frac{q_{k}}{2}\) for \(0\leq i<k\).
\item\label{item:q-estimates}
\(\sum_{1\leq j\leq k}q_{j-1}<2q_{k-1}\) for \(k\geq 1\), and
\(q_{k-1}\geq 2^{k}\) for \(k \geq 2\).
\end{enumerate}

For any \(k\in \N\), let
\[
{G_{k,m}}
 =
 {\left[
 mq_{k}+\frac{q_{k}}{2},
 mq_{k}+\frac{q_{k}}{2}+q_{k-1}-1
 \right]\cap \Zp}
 \ (m\in \Zp),\quad
 G_{k}=\bigcup_{m\in \Zp}{G_{k,m}},
\]
and
\begin{equation}
\label{eq:DA}
  G=\bigcup_{k\in \N} G_{k},
 \quad
 \Ascr=\Zp\setminus\! G.
\end{equation}
By the construction of $G_{k}$, it can be verified
\begin{equation}
\label{eq:Dk-structure}
 n\in G_{k} \text{ if and only if } n+q_{k}\in  G_{k}
 \ (n\in\Zp), \quad
 \#(G_{k,m})=q_{k-1} \ (m\in \Zp).
\end{equation}

\begin{lemma}
\(D_{*}(\Ascr)\geq \frac{1}{2}\).
\end{lemma}

\begin{proof}
For \(N\geq q_{3}\), by $\min G_{k}=\frac{q_{k}}{2}$,
only the sets \(G_{k}\) satisfying \(\frac{q_{k}}{2}\leq N\) meet
\([0,N]\). For each such \(k\), by the construction of $G_{k}$,
we get
\[
 \#(G_{k} \cap [0,N]) \leq
 \left\lceil\frac{N+1}{q_{k}}\right\rceil\cdot q_{k-1}
 \leq(N+1)\frac{q_{k-1}}{q_{k}}+q_{k-1},
\]
implying
\[
 \frac{\#(G\cap[0, N])}{N+1}
 \leq \sum_{1\leq k\leq K(N)}\frac{\#(G_{k}\cap[0,N])}{N+1}
 \leq
 \sum_{1\leq k\leq K(N)}\frac{q_{k-1}}{q_{k}}
 +\frac{\sum_{1\leq k\leq K(N)}q_{k-1}}{N+1},
\]
where \(K(N)=\max\{k\in \Zp: \frac{q_{k}}{2}\leq N\}\).
Noting that \(q_{K(N)-1}=\sqrt{\frac{q_{K(N)}}{2}} \leq \sqrt{N}\),
together with \textup{(\ref{item:q-estimates})}, we have
\begin{align*}
 \frac{\#(G\cap[0,N])}{N+1}
 \leq & \sum_{k\geq1}\frac{q_{k-1}}{q_{k}}
 +\frac{2q_{K(N)-1}}{N+1}
 \leq
 \frac{q_{0}}{q_{1}}+\sum_{k\geq2}\frac{q_{k-1}}{q_{k}}+
 \frac{2\sqrt{N}}{N+1}\\
 = &
 \frac{1}{4}+\sum_{k\geq2}\frac{1}{2q_{k-1}}
 + \frac{2\sqrt{N}}{N+1}
 \leq
 \frac{1}{4}+\sum_{k\geq2}\frac{1}{2^{k+1}}
 + \frac{2\sqrt{N}}{N+1}\\
 = & \frac{1}{2}+ \frac{2\sqrt{N}}{N+1},
\end{align*}
and thus $D^{*}(G)=\limsup_{N\to \infty}
 \frac{\#(G\cap[0,N])}{N+1}\leq \frac{1}{2}$.
Therefore
\[
D_{*}(\Ascr)=1-D^{*}(G)\geq \frac{1}{2}>0.
\]
\end{proof}

\begin{lemma}
\label{lem:H-not-ps}
\(\Ascr\notin \Fps\).
\end{lemma}

\begin{proof}
Suppose, on the contrary, that \(\Ascr\in \Fps\),
then there exists \(L\in\Zp\) such that
\(\bigcup_{i=0}^{L}(\Ascr-i)\) is thick.
This, together with $\bigcup_{i=0}^{L}(\Ascr-i)=
\bigcup_{i=0}^{L}(\Zp\setminus G-i)
=\bigcup_{i=0}^{L}(\Zp\setminus (G-i))
=\Zp\setminus\bigcap_{i=0}^{L}(G-i)
\subseteq \Zp\setminus
\bigcap_{i=0}^{L}(\bigcup_{m\in \N}{G_{L+1,m}}-i)
=\Zp\setminus
\bigcap_{i=0}^{L}(\bigcup_{m\in \N}({G_{L+1,m}}-i))$,
implies that $\bigcap_{i=0}^{L}(\bigcup_{m\in \N}
({G_{L+1,m}}-i))$ is not syndetic.
Meanwhile, noting that $\bigcap_{i=0}^{L}(\bigcup_{m\in \N}
({G_{L+1,m}}-i)) \supseteq
\bigcup_{m\in \N}(\bigcap_{i=0}^{L}({G_{L+1,m}}-i))
\supseteq \bigcup_{m\in \N}
[mq_{L+1}+\frac{q_{L+1}}{2},
 mq_{L+1}+\frac{q_{L+1}}{2}+q_{L}-1-L]
\supseteq \bigcup_{m\in \N}\{mq_{L+1}+\frac{q_{L+1}}{2}\}$
by \(q_{L}>L\), we have that
$\bigcap_{i=0}^{L}(\bigcup_{m\in \N}
({G_{L+1,m}}-i))$ is syndetic, which is a contradiction.
\end{proof}

\subsection{A subshift in $\boldsymbol{E_{0}\setminus M_{0}}$}
\label{sec:zero-entropy}

Choose a point \(a=(a_{i})_{i\in \Zp}\in \Sigma_2\)
as $a_{i}=1$ if and only if $i\in H$ and set
$Y=\overline{\orb(a,\sigma)}$.

We shall prove that \(a\) is an
\(\Fpubd\)-recurrent transitive point of the zero-entropy subshift
\((Y,\sigma|_{Y})\), that
\(N_{\sigma|_{Y}}(a,\Cyl{1}_{Y})\notin\Fps\), and consequently that
\((Y,\sigma|_{Y})\in\Ezero\setminus\Mzero\).

\begin{lemma}
\label{lem:zero-entropy}
Let $Y=\overline{\orb(a,\sigma)}$.
Then \(h_{\mathrm{top}}(\sigma|_{Y})=0.\)
\end{lemma}

\begin{proof}
Fix any \(n\in\N\). Let
\(K=\min\{k\in\Zp:q_{k}>4n\}\geq2\) and
\(\partial G_{j}=\{
mq_{j}+\frac{q_{j}}{2},
mq_{j}+\frac{q_{j}}{2}+q_{j-1}:m\in\Zp
\}\) for \(j>K\). From
\textup{(\ref{item:q-divisibility})} and \(q_{K}>4n\), it follows
\(\bigcup_{j>K}\partial G_{j}\subseteq q_{K}\Zp\) and
\(\#([s,s+n-1]\cap\bigcup_{j>K}\partial G_{j})\leq 1\)
($\forall s\in\Zp$).

For any \(s\in\Zp\), put
\[
L_{s}=\left\{i\in[0,n-1]:
s+i\in\bigcup_{1\leq k\leq K}G_{k}\right\}
\text{ and }
U_{s}=\left\{i\in[0,n-1]:
s+i\in\bigcup_{j>K}G_{j}\right\}.
\]

By \textup{(\ref{item:q-divisibility})} and
\eqref{eq:Dk-structure}, we have
\[
i\in\bigcup_{1\leq k\leq K}G_{k}
\text{ if and only if }
i+q_{K}\in\bigcup_{1\leq k\leq K}G_{k}
\quad (\forall i\in\Zp),
\]
implying \(L_{s+q_{K}}=L_{s}\ (\forall s\in\Zp).\)
Together with the choice of \(K\) and
\eqref{eq:q-definition}, we get
\begin{equation}\label{eq:Ls-Card}
\#(\{L_{s}:s\in\Zp\})\leq
q_{K}=2q_{K-1}^{2}\leq2(4n)^{2}=32n^{2}.
\end{equation}

For any \(1\leq i<n\) with
\(\#(\{i-1,i\}\cap U_{s})=1\), we consider the following two cases:
\begin{itemize}
  \item If \(i-1\notin U_{s}\) and \(i\in U_{s}\), then there exists \(j>K\)
  such that \(s+i-1\notin G_{j}\) and \(s+i\in G_{j}\), and thus
  \(s+i\in\partial G_{j}\);
  \item If \(i-1\in U_{s}\) and \(i\notin U_{s}\), then there exists \(j>K\)
  such that \(s+i-1\in G_{j}\) and \(s+i\notin G_{j}\), and thus
  \(s+i\in\partial G_{j}\).
\end{itemize}
In either case, \(s+i\in\bigcup_{j>K}\partial G_{j}\). Therefore
\(\#(\left\{i\in[1,n-1]:
\#(\{i-1,i\}\cap U_{s})=1\right\})
\leq\#(\big\{i\in[1,n-1]:
s+i\in\bigcup_{j>K}\partial G_{j}\big\})
\leq
\#([s,s+n-1]\cap\bigcup_{j>K}\partial G_{j})
\leq 1,\)
implying that
\begin{itemize}
  \item If $\#(\left\{i\in[1,n-1]:
\#(\{i-1,i\}\cap U_{s})=1\right\})=0$,
then $i-1\in U_{s}$ if and only if $i\in U_{s}$ for any $i\in [1, n-1]$,
and thus $U_{s}=\varnothing$ or $U_{s}=[0, n-1]$;
  \item If $\#(\left\{i\in[1,n-1]:
\#(\{i-1,i\}\cap U_{s})=1\right\})=1$, then
there exists $r\in[1,n-1]$ such that
$i-1\in U_{s}$ if and only if $i\in U_{s}$ for any
$i\in[1,n-1]\setminus\{r\}$, whereas
$r-1\in U_{s}$ if and only if $r\notin U_{s}$. Thus
$U_{s}=[0,r-1]$ or $U_{s}=[r,n-1]$.
\end{itemize}
In particular,
\(
U_{s}\in
\{\varnothing,[0,n-1]\}
\cup
\{[0,r-1],[r,n-1]:1\leq r<n\},
\)
so
\[
\#(\{U_{s}:s\in\Zp\})\leq2n.
\]
Together with \eqref{eq:Ls-Card},
noting that
\(
\{i\in[0, n-1]: a_{s+i}=0\}=
\{i\in[0, n-1]: s+i\in G\}
=L_{s}\cup U_{s} \ (\forall s\in \Zp),
\)
we have
\[
p_{a}(n)\leq
\#(\{L_{s}:s\in\Zp\})\cdot
\#(\{U_{s}:s\in\Zp\})\leq 64 n^{3}.
\]
This, together with \eqref{eq:entropy-sigma},
implies
\[
h_{\mathrm{top}}(\sigma|_{Y})
=\lim_{n\to\infty}\frac{\log p_{a}(n)}{n}=0.
\]
\end{proof}

\begin{lemma}
\label{lem:a-recurrent}
Let $Y=\overline{\orb(a,\sigma)}$.
Then \(a\) is an \(\Fpubd\)-recurrent point of
$\sigma|_{Y}$.
\end{lemma}

\begin{proof}
Fix any \(r\in \N\). By \(q_{j+1}\geq 2q_{j}\) for \(j\in\Zp\), we have that,
for any $K\in \N$,
\[
 \sum_{j>K}\frac{1}{q_{j}}
 \leq \frac{1}{2q_{K}^{2}}
 +\frac{1}{2^{2}q_{K}^{2}}+ \cdots
 +\frac{1}{2^{j}q_{K}^{2}}+\cdots
 =
 \frac{1}{q_{K}^{2}}.
\]
This, together with \textup{(\ref{item:q-estimates})}, implies
\begin{align*}
 & \sum_{j>K}\frac{q_{j-1}+r+q_{K}}{q_{j}}
 =
 \sum_{j>K}\frac{q_{j-1}}{q_{j}}
 +(r+q_{K})\sum_{j>K}\frac{1}{q_{j}}\\
 \leq &
 \sum_{j>K}\frac{1}{2q_{j-1}}
 +\frac{r+q_{K}}{q_{K}^{2}}
 \leq \frac{1}{2q_{K-1}^{2}}+
 \frac{r+q_{K}}{q_{K}^{2}}
 \to 0 \ (K\to +\infty).
\end{align*}
Then, there exists \(K_{0}\in \N\) such that
\begin{equation}
\label{eq:K-choice}
 \frac{q_{K_{0}+1}}{2}>r
 \text{ and } \sum_{j>K_{0}}
 \frac{q_{j-1}+r+q_{K_{0}}}{q_{j}}<\frac{1}{2}.
\end{equation}

For any \(j>K_{0}\), let
\(E_{j}=
 \{t\in q_{K_{0}}\Zp : [t,t+r-1]\cap G _{j}\neq\varnothing\},
\) $E^{r}=\bigcup_{j>K_{0}}E_{j}$, and \({R^{r}}=
 q_{K_{0}}\Zp\setminus E^{r}\).

\begin{claim}\label{clm:Rr-density}

\(D_{*}(R^{r})\geq\frac{1}{2q_{K_{0}}}.\)
\end{claim}

\textbf{Proof of Claim~\ref{clm:Rr-density}}.
By the construction of $E_{j}$, we have that,
for any $m\in \Zp$ and any $j>K_{0}$, $E_{j}\cap[m q_{j}, (m+1)q_{j}-1]
 \subseteq
 q_{K_{0}}\Zp\cap \left[
 m q_{j}+\frac{q_{j}}{2}-r+1,
 m q_{j}+\frac{q_{j}}{2}+q_{j-1}-1
 \right],$ and thus
 \begin{equation*}
 \#(E_{j}\cap[m q_{j}, (m+1)q_{j}-1])
 \leq \frac{q_{j-1}+r}{q_{K_{0}}}+1,
 \end{equation*}
 implying that, for any $N\in \N$,
 \begin{equation}\label{eq:Ej-cap-M-Card}
 \#(E_{j}\cap[0, N])
 \leq \left\lceil\frac{N}{q_{j}}\right\rceil\cdot
 \left(\frac{q_{j-1}+r}{q_{K_{0}}}+1\right)\leq
 {\left(\frac{N}{q_{j}}+1\right)}\cdot
 \left(\frac{q_{j-1}+r}{q_{K_{0}}}+1\right).
\end{equation}

Meanwhile, by the construction of $E_{j}$, there exist $0\leq i\leq r-1$ such that
$\min E_{j}+i\in G_{j}$, implying $\min E_{j}+(r-1)\geq \min E_{j}+i\geq \min G_{j}
=\frac{q_{j}}{2}$, and thus $\min E_{j}\geq \frac{q_{j}}{2}-(r-1)$. Therefore,
$E_{j}\cap [0, N]=\varnothing$ for $j>K_{0}$ such that $\frac{q_{j}}{2}-(r-1)>N$.
This, together with \textup{(\ref{item:q-estimates})}, \eqref{eq:K-choice}, and
\eqref{eq:Ej-cap-M-Card}, implies
\begin{align*}
 & \frac{1}{N+1}
 \#(E^{r}\cap[0, N])\\
 \leq & \frac{1}{N+1}
 \sum_{j>K_{0}}\#(E_{j}\cap[0, N])
 = \frac{1}{N+1}
 \sum_{\{j>K_{0}: E_{j}\cap [0, N]\neq \varnothing\}}
 \#(E_{j}\cap[0, N])\\
 \leq &
 \frac{1}{N+1}
 \sum_{\{j>K_{0}: q_{j}\leq 2(N+(r-1))\}}\#(E_{j}\cap[0, N])
 =\frac{1}{N+1}
 \sum_{K_{0}<j\leq J_{N}}\#(E_{j}\cap[0, N])\\
 \leq &
 \frac{1}{q_{K_{0}}}
 \sum_{K_{0}<j\leq J_{N}}\frac{q_{j-1}+r+q_{K_{0}}}{q_{j}}
 +\frac{1}{(N+1) q_{K_{0}}}
 \sum_{K_{0}<j\leq J_{N}}(q_{j-1}+r+q_{K_{0}})\\
 \leq & \frac{1}{2q_{K_{0}}}
 +\frac{1}{(N+1) q_{K_{0}}}
 \sum_{K_{0}<j\leq J_{N}}q_{j-1} (1+r+q_{K_{0}})
 \leq \frac{1}{2q_{K_{0}}}+
 \frac{2(1+r+q_{K_{0}})q_{J_{N}-1}}{(N+1) q_{K_{0}}}\\
 \leq & \frac{1}{2q_{K_{0}}}+
 \frac{2(1+r+q_{K_{0}})\sqrt{N+(r-1)}}{(N+1) q_{K_{0}}},
\end{align*}
where $J_{N}=\max\{j>K_{0}: q_{j}\leq 2(N+(r-1))\}$.
Thus
\[
D^{*}(E^{r})=\limsup_{N\to \infty}\frac{1}{N+1}
 \#(E^{r}\cap[0, N])\leq \frac{1}{2q_{K_{0}}}.
 \]
Hence
\[
 D_{*}(R^{r})
 \geq\frac{1}{q_{K_{0}}}
 -D^{*}(E^{r})\geq
 \frac{1}{2q_{K_{0}}}.
\]

\begin{claim}\label{clm:return-inclusion}
$R^{r}\cap\N \subseteq
N_{\sigma|_{Y}}(a,\Cyl{a_{0}a_{1}\cdots a_{r-1}}_{Y})$.
\end{claim}

\textbf{Proof of Claim~\ref{clm:return-inclusion}}.
For any \(t\in R^{r}\cap \N\), by
\(R^{r}\subseteq q_{K_{0}}\Zp\), \textup{(\ref{item:q-divisibility})
gives \(q_{k}| t\) for \(1\leq k\leq K_{0}\).}  Together with
\eqref{eq:Dk-structure}, we have that, for any
$0\leq i<r$ and any $1\leq k\leq K_{0}$, $i\in G_{k}$ if
and only if $t+i\in G_{k}$, and thus
\begin{equation}\label{eq:t+i-equivalent-i}
t+i\in\bigcup_{1\leq k\leq K_{0}} G_{k}
\text{ if and only if }
i\in\bigcup_{1\leq k\leq K_{0}} G_{k}.
\end{equation}

On the other hand,
$\frac{q_{K_{0}+1}}{2}>r$ in \eqref{eq:K-choice}, together with
$\min G_{k}=\frac{q_{k}}{2}$, implies
\([0,r-1]\cap\bigcup_{k>K_{0}} G_{k}=\varnothing,\)
and the definition of \(R^{r}\) gives
\([t,t+r-1]\cap\bigcup_{k>K_{0}} G_{k}=\varnothing.\)
This, together with \eqref{eq:t+i-equivalent-i}
and $G=\bigcup_{k\in \N}G_{k}$, yields that,
for any $0\leq i<r$, $t+i\in G$ if and only if $i\in G$.
Thus, by the choice of $a$, we obtain
\(a_{t}a_{t+1}\cdots a_{t+r-1}=a_{0}a_{1}\cdots a_{r-1}\),
implying $(\sigma|_{Y})^{t}(a)\in \Cyl{a_{0}a_{1}\cdots a_{r-1}}_{Y}$,
i.e., $t\in N_{\sigma|_{Y}}(a, \Cyl{a_{0}a_{1}\cdots a_{r-1}}_{Y})$.
Therefore
\[
R^{r}\cap \N \subseteq
N_{\sigma|_{Y}} (a,\Cyl{a_{0}a_{1}\cdots a_{r-1}}_{Y}).
\]

Summing up Claims~\ref{clm:Rr-density} and \ref{clm:return-inclusion},
for any $r\in \N$,
\[
 BD^{*}(
 N_{\sigma|_{Y}}(a, \Cyl{a_{0}a_{1}\cdots a_{r-1}}_{Y}))
 \geq D_{*}( N_{\sigma|_{Y}}(a, \Cyl{a_{0}a_{1}\cdots a_{r-1}}_{Y}))
 \geq D_{*}({R^{r}}\cap\N)
 =D_{*}({R^{r}})>0,
\]
implying that
\(a\) is \(\Fpubd\)-recurrent.
\end{proof}

\begin{remark}
\label{rem:Y-system}
Lemma~\ref{lem:H-not-ps} implies that \(H\) is not syndetic. Then,
for any $n\in \N$, there exists $i_{n}\in \N$ such that $[i_{n}, i_{n}+n]\subseteq
\Zp\setminus H$, implying $a_{i_{n}}=\cdots =a_{i_{n}+n}=0$, and thus
$0^{\infty}=\lim_{n\to \infty} \sigma^{i_{n}}(a)$.
Moreover, the point \(0^{\infty}\) is
the unique minimal point of \((Y,\sigma|_{Y})\). Indeed,
suppose that there exists a minimal point \(z\in Y\setminus \{0^{\infty}\}\).
Replacing \(z\) by a shift of it, we may assume that \(z_{0}=1\).
Noting that $\Cyl{1}_{Y}$ is a neighborhood of $z$,
by the minimality of $z$, there exists \(L\in\N\) such
that
\[
 [n,n+L]\cap N_{\sigma|_{Y}}(z,\Cyl{1}_{Y})
 \neq \varnothing
 \quad (\forall n\in\Zp).
\]
Meanwhile, for any \(N \in \N\), since \(z\in Y=\overline{\orb(a,\sigma)}\),
there exists \(s_{N} \in \Zp\) such that, for any \(0\leq t\leq N+L\),
\(a_{s_{N}+t}=z_{t}\). Therefore, for any \(0\leq n\leq N\),
there exists \(0\leq i\leq L\) such that
\(s_{N}+n+i\in H\), and hence
\[
 [s_{N},s_{N}+N]\cap \Zp\subseteq \bigcup_{i=0}^{L}(H-i)
 \quad (\forall N\in \N),
\]
i.e., \(\bigcup_{i=0}^{L}(H-i)\) is thick.
Thus $H\in \Fps$, contradicting
Lemma~\ref{lem:H-not-ps}.
\end{remark}

\begin{theorem}
\label{thm:Y-system}
Let $Y=\overline{\orb(a,\sigma)}$.
Then the subshift $(Y, \sigma|_{Y})$ satisfies the following conditions:
\begin{enumerate}
\renewcommand{\labelenumi}{\textup{(\roman{enumi})}}
\renewcommand{\theenumi}{\roman{enumi}}
\item\label{item:Y-return-set}
$N_{\sigma|_{Y}}(a, \Cyl{1}_{Y})
=\Ascr\cap\N \notin \Fps$ for the
open neighborhood $\Cyl{1}_{Y}$ of $a$.
\item\label{item:Y-class} $(Y,\sigma|_{Y})\in
       \Ezero\setminus\Mzero
       \subseteq
       \Ezero\setminus\mathscr{M}_{0}$.
\end{enumerate}
\end{theorem}

\begin{proof}
(i) Clearly $\Cyl{1}_{Y}$ is an open neighborhood of $a$.
By the construction of $a$, Lemma~\ref{lem:H-not-ps} gives
\begin{equation}
\label{eq:a-return-set}
 N_{\sigma|_{Y}}(a,\Cyl{1}_{Y})=\{n\in \N:
 ((\sigma|_{Y})^{n}(a))_{0}=1\}=\Ascr\cap\N
 \notin \Fps.
\end{equation}

(ii)
Clearly $a\in\operatorname{Tran}(Y,\sigma|_{Y})$.
Lemma~\ref{lem:a-recurrent},
Lemma~\ref{lem:ME}~\textup{(\ref{item:ME-pubd})},
and Lemma~\ref{lem:zero-entropy} give
\((Y,\sigma|_{Y})\in\Ezero.\)
Meanwhile, by \eqref{eq:a-return-set},
Lemma~\ref{lem:ME}~\textup{(\ref{item:ME-ps})} shows that
\((Y,\sigma|_{Y})\) is not an \(M\)-system. Therefore
\((Y,\sigma|_{Y})\in\Ezero\setminus\Mzero
\subseteq
\Ezero\setminus\mathscr{M}_{0}.\)
\end{proof}

\subsection{A minimal subshift with u.p.e.}
\label{sec:upe}

This subsection constructs a minimal subshift
\((X,\sigma|_{X})\) with u.p.e.\ and a transitive point
\({\hat{x}}\in X\) satisfying \({\hat{x}_{0}}=1\) and
\(\{n\in\N:{\hat{x}_{n}}=1\}\subseteq G\). We also prove that
\((X,\sigma|_{X})\) is not disjoint from the subshift
\((Y,\sigma|_{Y})\) constructed in the preceding subsection.

For a finite collection \(\mathscr{W}\) of finite words
on $\Sigma$ and \(n\in \Zp\), let
\[
 \mathscr{C}_{\mathscr{W}}(n)=
 \left\{{w_{1}\cdots w_{s}}:
 s\in \Zp,\ w_{i}\in\mathscr{W}\ (1\leq i\leq s),\
 \sum_{i=1}^{s}|w_{i}|=n
 \right\},
\]
and
\[
\mathscr{C}_{\mathscr{W}}=\bigcup_{n\in \Zp}
\mathscr{C}_{\mathscr{W}}(n)
=\left\{{w_{1}\cdots w_{s}}:
s\in \Zp,\ w_{i}\in\mathscr{W}\ (1\leq i\leq s)\right\}.
\]
It is easy to see that if $\mathscr{W}_{1}\subseteq \mathscr{W}_{2}$, then for any $n\in \Zp$,
$\mathscr{C}_{\mathscr{W}_{1}}(n)\subseteq \mathscr{C}_{\mathscr{W}_{2}}(n)$
and $\mathscr{C}_{\mathscr{W}_{1}}\subseteq \mathscr{C}_{\mathscr{W}_{2}}$.

\begin{lemma}
\label{lem:alignment}
Let \(\mathscr{W}\) be a finite collection of finite words
containing a word of length \(L\) and a word of length \(L+1\)
for some $L\in \N$.  If words \(u\) and \(v\) of the
same length occur in members of \(\mathscr{W}\), then
there exist \(d\in\N\), \(0\leq\alpha\leq d-|u|\),
and \({c^{(u)},c^{(v)}}\in\mathscr{C}_{\mathscr{W}}(d)\)
such that
\[
 {(c^{(u)})_{\alpha}(c^{(u)})_{\alpha+1}\cdots
 (c^{(u)})_{\alpha+|u|-1}=u}
 \text{ and }
 {(c^{(v)})_{\alpha}(c^{(v)})_{\alpha+1}\cdots
 (c^{(v)})_{\alpha+|u|-1}=v}.
\]
\end{lemma}

\begin{proof}
Fix \(w_{L}, w_{L+1}\in\mathscr{W}\) with
\(|w_{L}|=L\) and \(|w_{L+1}|=L+1\). By hypothesis,
there exist \(w^{(u)}, w^{(v)}\in\mathscr{W}\) and
\(\alpha_{u}, \alpha_{v}\in\Zp\) such that
\((w^{(u)})_{\alpha_{u}}(w^{(u)})_{\alpha_{u}+1}\cdots
 (w^{(u)})_{\alpha_{u}+|u|-1}=u\) and
 \((w^{(v)})_{\alpha_{v}}(w^{(v)})_{\alpha_{v}+1}$
 $\cdots
 (w^{(v)})_{\alpha_{v}+|v|-1}=v.\)
By \(\gcd(L,L+1)=1\), Sylvester's theorem
\cite[Theorem~2.1.1]{RamirezAlfonsin} gives
\([L^{2}-L, +\infty)\cap\Zp
 \subseteq
 \{rL+s(L+1): r,s\in\Zp\}.\)
In particular, for
\(\alpha= L^{2}-L+\max\{\alpha_{u},\alpha_{v}\}+1\),
there exist
\(r_{u},s_{u},r_{v},s_{v}\in\Zp\) satisfying
\(\alpha-\alpha_{u}=r_{u}L+s_{u}(L+1)\) and
\(\alpha-\alpha_{v}=r_{v}L+s_{v}(L+1).\)
Then, take
\(\tilde{c}^{(u)}=w_{L}^{r_{u}}w_{L+1}^{s_{u}}w^{(u)}\) and
\(\tilde{c}^{(v)}=w_{L}^{r_{v}}w_{L+1}^{s_{v}}w^{(v)}.\)
Clearly,
\begin{equation}
\label{eq:Constructions-Cu-Cv}
 {(\tilde{c}^{(u)})_{\alpha}(\tilde{c}^{(u)})_{\alpha+1}\cdots
 (\tilde{c}^{(u)})_{\alpha+|u|-1}=u}
 \text{ and }
 {(\tilde{c}^{(v)})_{\alpha}(\tilde{c}^{(v)})_{\alpha+1}\cdots
 (\tilde{c}^{(v)})_{\alpha+|u|-1}=v}.
\end{equation}

Next, choose \(d= L^{2}-L+\max\{|\tilde{c}^{(u)}|, |\tilde{c}^{(v)}|\}
+\alpha\geq \alpha+|u|.\)
Then there exist
\(r'_{u},s'_{u},r'_{v},s'_{v}\in\Zp\) satisfying
\(d-|\tilde{c}^{(u)}|=r'_{u}L+s'_{u}(L+1)\)
and \(d-|\tilde{c}^{(v)}|=r'_{v}L+s'_{v}(L+1)\).
This, together with the constructions
of $\tilde{c}^{(u)}$ and $\tilde{c}^{(v)}$
and \eqref{eq:Constructions-Cu-Cv}, implies
\(c^{(u)}=\tilde{c}^{(u)} w_{L}^{r'_{u}}w_{L+1}^{s'_{u}}\in
\mathscr{C}_{\mathscr{W}}(d)\),
\(c^{(v)}=\tilde{c}^{(v)} w_{L}^{r'_{v}}w_{L+1}^{s'_{v}}\in
\mathscr{C}_{\mathscr{W}}(d)\), and
\({(c^{(u)})_{\alpha}(c^{(u)})_{\alpha+1}\cdots
 (c^{(u)})_{\alpha+|u|-1}=u}\),
\({(c^{(v)})_{\alpha}(c^{(v)})_{\alpha+1}\cdots
 (c^{(v)})_{\alpha+|u|-1}=v}\).
\end{proof}

\textcolor{red}{The following proposition uses this alignment lemma to
construct the finite words underlying the required subshift.}

\begin{proposition}
\label{prop:block-construction}
There exist strictly increasing \(L_{m}\in\N\), \(\mathcal{V}_{m}
\subseteq \{0,1\}^{L_{m}}\), \(P_{m},A_{m}\in\{0,1\}^{L_{m}}\),
\(B_{m}\in\{0,1\}^{L_{m}+1}\), and finite collections \(\mathscr{P}_{m}\)
of pairs \(\pi=\{u_{\pi},v_{\pi}\}\) of distinct words of the same
length and sets \(I_{\pi,m}\) with
\begin{equation}
\label{eq:P-m-and-Ipim}
 \mathscr{P}_{m}\subseteq \mathscr{P}_{m+1}\ (m \in \N),
 \quad
 \varnothing\neq I_{\pi,m}\subseteq
 \{0,\ldots,L_{m}-|u_{\pi}|\}
 \ ({m \geq 2}, \pi\in\mathscr{P}_{m}),
\end{equation}
satisfying, for each \(m\in\N\), the following conditions:
\begin{enumerate}
\renewcommand{\theenumi}{C\arabic{enumi}}
\renewcommand{\labelenumi}{\textup{(\theenumi)}}
\item
\label{cond:concatenation}
$\mathscr{W}_{m+1}\subseteq \mathscr{C}_{\mathscr{W}_{m}}$
and $\mathscr{W}_{m}\subseteq \bigcap_{w\in \mathscr{W}_{m+1}}
\Sub(w)$, where
\(\mathscr{W}_{m}=\mathcal{V}_{m}\cup\{P_{m},A_{m},B_{m}\}\) and
\(\mathscr{W}_{m+1}=\mathcal{V}_{m+1}
\cup\{P_{m+1}, A_{m+1}, B_{m+1}\}\).
\item
\label{cond:prefix}
\((A_{m+1})_{0}(A_{m+1})_{1}\cdots(A_{m+1})_{L_{m}-1}=A_{m}\).
\item
\label{cond:support}
\(\{i\in\Zp:(P_{m})_{i\bmod L_{m}}=1\}\subseteq G\),
where \(G\) is the set defined in \eqref{eq:DA}.
\item
\label{cond:coordinates}
\((P_{m})_{0}=0\), \((A_{m})_{0}=1\), and
 \((A_{m})_{i}=(P_{m})_{i}\) \((1\leq i<L_{m})\).
\item
\label{cond:pairs}
 \(\{\{u,v\}:u,v\in\Sub(A_{m}),\ u\neq v,\
 1\leq |u|=|v|\leq m\}\subseteq \mathscr{P}_{m+1}\).
\item
\label{cond:independence}
{If \(m\geq2\), then for each
\(\pi=\{u_{\pi}, v_{\pi}\}\in\mathscr{P}_{m}\) and each map}
\(\eta\colon I_{\pi,m}\to\{u_{\pi},v_{\pi}\}\), there exists
\(w\in\mathcal{V}_{m}\) such that
\[
w_{i}w_{i+1}\cdots w_{i+|u_{\pi}|-1}=\eta(i)
 \quad (i\in I_{\pi,m}).
\]
\item
\label{cond:propagation}
{If \(m\geq2\), then for each
\(\pi\in\mathscr{P}_{m}\),}
\[
 \frac{\#(I_{\pi,m+1})}{L_{m+1}}
 \geq
 (1-2^{-m-4})\frac{\#(I_{\pi,m})}{L_{m}}.
 \]
\end{enumerate}
\end{proposition}

\begin{proof}
We construct the stated objects inductively.

For \(m=1\), put
\[
 k_{1}=1,\quad L_{1}=q_{1}=4,\quad
 P_{1}=0010,\quad A_{1}=1010,\quad B_{1}=00000,\quad
 \mathcal{V}_{1}=\{0,1\}^{4},\quad \mathscr{P}_{1}=\varnothing.
\]
Then
\[
  \{i\in\Zp:(P_{1})_{i\bmod L_{1}}=1\}=G_{1},\quad
  (P_{1})_{0}=0,\quad (A_{1})_{0}=1,\quad
  (A_{1})_{i}=(P_{1})_{i}\ (1\leq i<L_{1}),
\]
and so conditions (\ref{cond:support}) and
(\ref{cond:coordinates}) hold for \(m=1\).

\medskip

Assume that, for some \(m\in\N\), all the objects with indices at
most \(m\) have been constructed. First, enumerate
\(\mathscr{W}_{m}=\{w_{1},\ldots,w_{\#(\mathscr{W}_{m})}\}\), and set
\[
 K_{m}=w_{1}w_{2}\cdots w_{\#(\mathscr{W}_{m})},\
 \mathscr{Q}_{m}= \{\{u,v\}:u,v\in\Sub(A_{m}),\ u\neq v,\
 1\leq |u|=|v|\leq m\}.
\]
It can be verified that
\begin{itemize}
  \item \(\mathscr{W}_{m}\subseteq \Sub(K_{m}).\)
  \item For any \(\pi\in\mathscr{Q}_{m}\), by Lemma~\ref{lem:alignment},
  there exist \(d_{\pi}\in\N\), \(0\leq\alpha_{\pi}\leq d_{\pi}-|u_{\pi}|\),
  and \(c_{\pi}^{0}\), \(c_{\pi}^{1}\in \mathscr{C}_{\mathscr{W}_{m}}(d_{\pi})\)
such that
\[
\begin{cases}
 (c_{\pi}^{0})_{\alpha_{\pi}}(c_{\pi}^{0})_{\alpha_{\pi}+1}\cdots
 (c_{\pi}^{0})_{\alpha_{\pi}+|u_{\pi}|-1}=u_{\pi},\\
 (c_{\pi}^{1})_{\alpha_{\pi}}(c_{\pi}^{1})_{\alpha_{\pi}+1}\cdots
 (c_{\pi}^{1})_{\alpha_{\pi}+|u_{\pi}|-1}=v_{\pi}.
\end{cases}
\]
\end{itemize}

Since \(q_{k}\), \(q_{k-1} \to +\infty\) as \(k\to +\infty\)
(by \eqref{eq:q-definition}) and \(\mathscr{Q}_{m}\) is finite,
there exists \(k_{m+1}\geq k_{m}+2\) such that
\({q_{k_{m+1}-1}-|K_{m}|\geq L_{m}^{2}-L_{m}}\),
\({|K_{m}|\leq 2^{-m-7} q_{k_{m+1}}}\),
\(L_{m}^{2}\leq 2^{-m-7} q_{k_{m+1}}\), and
\(\frac{2^{-m-8}q_{k_{m+1}}}{\#(\mathscr{Q}_{m})d_{\pi}}
\geq 1\) (\(\forall \pi \in \mathscr{Q}_{m}\)).
Choose \(L_{m+1}=q_{k_{m+1}}\). Clearly $L_{m+1}>L_{m}$.
Moreover, Sylvester's theorem
\cite[Theorem~2.1.1]{RamirezAlfonsin} gives
\([L_{m}^{2}-L_{m},+\infty)\cap\Zp
 \subseteq \{rL_{m}+s(L_{m}+1):r,s\in\Zp\}.\)
Noting that \({q_{k_{m+1}-1}-|K_{m}|\geq L_{m}^{2}-L_{m}}\)
and \(L_{m+1}+1-|K_{m}|
 \geq(1-2^{-m-7})L_{m+1}
 \geq 2^{-m-7}L_{m+1}
 \geq L_{m}^{2}>L_{m}^{2}-L_{m}\), together with
 $|P_{m}|=L_{m}$ and $|B_{m}|=L_{m}+1$, we
 have that there exist
\({F_{m}^{(P)}}\in\mathscr{C}_{\{P_{m},B_{m}\}}
 (q_{k_{m+1}-1}-|K_{m}|)\) and
\({F_{m}^{(B)}}\in\mathscr{C}_{\{P_{m},B_{m}\}}
 (L_{m+1}+1-|K_{m}|).\)
According to \(L_{m} | \frac{L_{m+1}}{2}\) and
\(L_{m} | q_{k_{m+1}-1}\), we
define
\begin{align*}
 P_{m+1}
 = & P_{m}^{{\frac{L_{m+1}}{2L_{m}}}}K_{m}{F_{m}^{(P)}}
  P_{m}^{{\frac{L_{m+1}/2-q_{k_{m+1}-1}}{L_{m}}}},\\
A_{m+1}
 = & A_{m}P_{m}^{{\frac{L_{m+1}}{2L_{m}}}-1}K_{m} {F_{m}^{(P)}}
  P_{m}^{{\frac{L_{m+1}/2-q_{k_{m+1}-1}}{L_{m}}}},\\
B_{m+1}= & K_{m} {F_{m}^{(B)}}.
\end{align*}
It is easy to see that
\begin{equation}
\label{eq:P-A-B-Def}
P_{m+1}, A_{m+1}\in \{0, 1\}^{L_{m+1}}
\cap \mathscr{C}_{\mathscr{W}_{m}}
\text{ and }
B_{m+1}\in \{0, 1\}^{L_{m+1}+1}
\cap \mathscr{C}_{\mathscr{W}_{m}}.
\end{equation}

Next, for each \(\pi\in\mathscr{Q}_{m}\), set
\begin{equation}
\label{eq:t-pi-sm}
 t_{\pi}=
 \left\lfloor
 \frac{2^{-m-7}L_{m+1}}{\#(\mathscr{Q}_{m})d_{\pi}}
 \right\rfloor
 \text{ and }
 S_{m}=|K_{m}|+\sum_{\pi\in\mathscr{Q}_{m}}t_{\pi} d_{\pi}.
\end{equation}
Then
\begin{equation}
\label{eq:t-estimates}
\begin{cases}
 \displaystyle{t_{\pi}\geq
 \frac{2^{-m-7}L_{m+1}}{\#(\mathscr{Q}_{m})d_{\pi}}-1\geq
 \frac{2^{-m-8}L_{m+1}}{\#(\mathscr{Q}_{m})d_{\pi}}
 \geq 1,} \\
 \displaystyle{\sum_{\pi\in\mathscr{Q}_{m}}t_{\pi} d_{\pi}
 \leq \sum_{\pi\in\mathscr{Q}_{m}}
 \frac{2^{-m-7}L_{m+1}}{\#(\mathscr{Q}_{m})d_{\pi}}d_{\pi}
 = {2^{-m-7}L_{m+1}}.}
 \end{cases}
\end{equation}
Choose the unique integer \(h_{m}\in \N\) satisfying
\(L_{m}^{2}-L_{m}\leq h_{m}<L_{m}^{2}\) and
\(h_{m}\equiv L_{m+1}-S_{m}\pmod{L_{m}},\)
and put
\begin{equation}\label{eq:xi-m-choice}
\xi_{m}=\frac{{L_{m+1}}-S_{m}-h_{m}}{L_{m}}\in \mathbb{Z}.
\end{equation}
Then, by \eqref{eq:t-estimates},
$|K_{m}|\leq 2^{-m-7}L_{m+1}$, and $h_{m}
<L_{m}^{2}\leq 2^{-m-7}L_{m+1}$, we get
\begin{equation}
\label{eq:xi-m-estimate}
 \xi_{m} L_{m}={L_{m+1}}-S_{m}-h_{m}
 \geq {(1-3\cdot2^{-m-7}) L_{m+1}}
 \geq {(1-2^{-m-4}) L_{m+1}}>0,
\end{equation}
implying \(\xi_{m}\in\N\) by \eqref{eq:xi-m-choice}.
Moreover, by \(h_{m}\in\{rL_{m}+s(L_{m}+1):r,s\in\Zp\}\),
\(|P_{m}|=L_{m}\), and \(|B_{m}|=L_{m}+1\), there exists
\(F_{m}\in\mathscr{C}_{\{P_{m}, B_{m}\}}(h_{m})
\subseteq \mathscr{C}_{\mathscr{W}_{m}}(h_{m})\).
\medskip

Then, enumerate
\(\mathscr{Q}_{m}=\{\pi_{1},\ldots,\pi_{\#(\mathscr{Q}_{m})}\}\), and set
\[
 \mathcal{V}_{m+1}
 =\left\{K_{m}
 \Bigg(
 \prod_{a=1}^{\#(\mathscr{Q}_{m})}
 \prod_{\ell=1}^{t_{\pi_{a}}}
 c_{\pi_{a}}^{\zeta_{\pi_{a},\ell}}
 \Bigg)
 v^{(1)}\cdots v^{(\xi_{m})}F_{m}:
 \begin{array}{l}
 \zeta_{\pi_{a},\ell}\in\{0,1\}
 \ (1\leq a\leq\#(\mathscr{Q}_{m}), 1\leq\ell\leq t_{\pi_{a}}),\\
 v^{(j)}\in\mathcal{V}_{m}\ (1\leq j\leq \xi_{m})
 \end{array}
 \right\}.
\]
Then $\mathcal{V}_{m+1}\subseteq \{0, 1\}^{L_{m+1}}\cap \mathscr{C}_{\mathscr{W}_{m}}$ by
\(|K_{m}|+{\sum_{a=1}^{\#(\mathscr{Q}_{m})}t_{\pi_{a}}d_{\pi_{a}}}
 +\xi_{m} L_{m}+ |F_{m}|={L_{m+1}}\) according to \eqref{eq:t-pi-sm}
 and \eqref{eq:xi-m-choice}.
Meanwhile, let
\begin{equation}
\label{eq:s-pi-a-ell}
 s_{\pi_{a},\ell}
 =|K_{m}|+{\sum_{b=1}^{a-1}t_{\pi_{b}}d_{\pi_{b}}}
 +(\ell-1)d_{\pi_{a}}
 \quad
 (1\leq a\leq\#(\mathscr{Q}_{m}),\ 1\leq\ell\leq t_{\pi_{a}}),
\end{equation}
and
\begin{equation}
\label{eq:rj-and-Pm+1}
 r_{j}=S_{m}+(j-1)L_{m}
 \quad(1\leq j\leq\xi_{m}),\quad
 \mathscr{P}_{m+1}=\mathscr{P}_{m}\cup
 (\mathscr{Q}_{m}\setminus \mathscr{P}_{m}).
\end{equation}

Finally, for any \(\pi\in\mathscr{P}_{m+1}\), define
\[
I_{\pi,m+1}=
\begin{cases}
\bigcup_{j=1}^{\xi_{m}}(r_{j}+I_{\pi,m}), & \pi\in\mathscr{P}_{m}, \\
\{s_{\pi,\ell}+\alpha_{\pi}: 1\leq\ell\leq t_{\pi}\}, &
\pi\in\mathscr{Q}_{m}\setminus\mathscr{P}_{m}.
\end{cases}
\]
Clearly $I_{\pi, 2}\neq \varnothing$ for any $\pi\in \mathscr{P}_2
=\mathscr{Q}_1\setminus \mathscr{P}_1$ by $\mathscr{P}_1=\varnothing$
and \(t_{\pi}\geq 1\) (see~\eqref{eq:t-estimates}). By the inductive
definitions of \(I_{\pi,m+1}\) and \(\mathscr{P}_{m+1}\), and
\(t_{\pi}\geq 1\) for \(\pi\in\mathscr{Q}_{m}\), it is easy to see that
\(I_{\pi,m+1}\neq\varnothing\) for any
\(\pi\in\mathscr{P}_{m+1}\).
Now, we first verify $I_{\pi, m+1}\subseteq
\{0,\ldots,L_{m+1}-|u_{\pi}|\}$.

If \(\pi\in\mathscr{P}_{m}\), according to
\(r_{j+1}-r_{j}=L_{m}\) by \eqref{eq:rj-and-Pm+1}
and \(I_{\pi,m}\subseteq \{0,\ldots,L_{m}-|u_{\pi}|\},\)
the sets \(r_{j}+I_{\pi,m}\), \(1\leq j\leq\xi_{m}\), are
pairwise disjoint. Thus
\begin{equation}
\label{eq:I-propagation-cardinality}
 \#(I_{\pi,m+1})
 =\sum_{j=1}^{\xi_{m}}\#(r_{j}+I_{\pi,m})
 =\xi_{m}\#(I_{\pi,m})
 \quad (\pi \in \mathscr{P}_{m}).
\end{equation}
Moreover, $\max I_{\pi, m+1}=r_{\xi_{m}}+\max I_{\pi, m}
 \leq r_{\xi_{m}}+L_{m}-|u_{\pi}|
 =S_m+\xi_{m}L_{m}-|u_{\pi}|
 =L_{m+1}-h_{m}-|u_{\pi}|
 \leq L_{m+1}-|u_{\pi}|$, implying
 $I_{\pi, m+1}\subseteq
 \{0,\ldots,L_{m+1}-|u_{\pi}|\}$.

If \(\pi\in\mathscr{Q}_{m}\setminus\mathscr{P}_{m}\),
by \eqref{eq:t-pi-sm}, \eqref{eq:xi-m-estimate}, and
\eqref{eq:s-pi-a-ell},
then
\[
 \max I_{\pi,m+1}
 =s_{\pi,t_{\pi}}+\alpha_{\pi}
 \leq s_{\pi,t_{\pi}}+d_{\pi}-|u_{\pi}|
 \leq S_{m}-|u_{\pi}|
 \leq L_{m+1}-|u_{\pi}|,
\]
implying
\(I_{\pi,m+1}\subseteq \{0,\ldots,L_{m+1}-|u_{\pi}|\}\).

\smallskip

Next, we verify conditions
(\ref{cond:concatenation})--(\ref{cond:propagation}):

\smallskip

(C1) By \eqref{eq:P-A-B-Def} and the construction of
$\mathcal{V}_{m+1}$,
\[
 \mathscr{W}_{m+1}=\mathcal{V}_{m+1}\cup\{P_{m+1}, A_{m+1}, B_{m+1}\}
 \subseteq\mathscr{C}_{\mathscr{W}_{m}},
 \quad
 \mathscr{W}_{m}\subseteq \Sub(K_{m})
 \subseteq\bigcap_{w\in\mathscr{W}_{m+1}}\Sub(w).
\]

(C2) By $A_{m+1}
 = A_{m}P_{m}^{{\frac{L_{m+1}}{2L_{m}}}-1}K_{m} {F_{m}^{(P)}}
  P_{m}^{{\frac{L_{m+1}/2-q_{k_{m+1}-1}}{L_{m}}}}$,
\[
 (A_{m+1})_{0}(A_{m+1})_{1}\cdots(A_{m+1})_{L_{m}-1}=A_{m}.
\]

(C3) By
\(P_{m+1}
 =P_{m}^{\frac{L_{m+1}}{2L_{m}}}K_{m}F_{m}^{(P)}
  P_{m}^{\frac{L_{m+1}/2-q_{k_{m+1}-1}}{L_{m}}},\)
\(\big|P_{m}^{\frac{L_{m+1}}{2L_{m}}}\big|
 =\frac{L_{m+1}}2,\)
\(|K_{m}F_{m}^{(P)}|=q_{k_{m+1}-1},\) and
\(G_{k_{m+1},0}
 =\left[\frac{L_{m+1}}2,
  \frac{L_{m+1}}2+q_{k_{m+1}-1}-1\right]\),
we get
\[
 (P_{m+1})_{i}=(P_{m})_{i\bmod L_{m}}
 \quad
 (i\in\{0,\ldots,L_{m+1}-1\}
 \setminus G_{k_{m+1},0}).
\]
This, together with $L_{m}\mid L_{m+1}$ and
\(G_{k_{m+1}}=\bigcup_{r\in\Zp}
 (rL_{m+1}+G_{k_{m+1},0}),\) implies
\[
 (P_{m+1})_{i\bmod L_{m+1}}
 =(P_{m})_{i\bmod L_{m}}
 \quad(i\in\Zp\setminus G_{k_{m+1}}).
\]
Therefore, by the induction hypothesis,
\[
 \{i\in\Zp:(P_{m+1})_{i\bmod {L_{m+1}}}=1\}
 \subseteq
 \{i\in\Zp:(P_{m})_{i\bmod L_{m}}=1\}\cup G_{k_{m+1}}
 \subseteq G\cup G_{k_{m+1}}
 = G.
\]

(C4) The definitions of \(P_{m+1}\) and \(A_{m+1}\), together with the
induction hypothesis (C4), give
\[
 (P_{m+1})_{0}=0,\quad (A_{m+1})_{0}=1,\quad
 (A_{m+1})_{i}=(P_{m+1})_{i}\quad
 (1\leq i< L_{m+1}).
\]

(C5) By the definitions of \(\mathscr{Q}_{m}\) and \(\mathscr{P}_{m+1}\),
\[
 \{\{u,v\}:u,v\in\Sub(A_{m}),\ u\neq v,\
 1\leq|u|=|v|\leq m\}
 =\mathscr{Q}_{m}\subseteq \mathscr{P}_{m} \cup
 (\mathscr{Q}_{m}\setminus \mathscr{P}_m)=
 \mathscr{P}_{m+1}.
\]

(C6) First, it is not difficult to check that condition~(\ref{cond:independence})
holds for $m=2$. Next, let $m\geq 2$. For each $\pi=\{u_{\pi}, v_{\pi}\}
\in \mathscr{P}_{m+1}$ and each map $\eta: I_{\pi, m+1}\to
\{u_{\pi}, v_{\pi}\}$,
by the definition of $I_{\pi, m+1}$, we consider
the following two cases:

C6-1) If $\pi \in \mathscr{P}_m$, by $I_{\pi, m+1}
=\bigcup_{j=1}^{\xi_{m}}(r_{j}+I_{\pi,m})$,
for any $1\leq j\leq \xi_{m}$, we define the map
$\eta_{j}: I_{\pi, m}\to \{u_{\pi}, v_{\pi}\}$ by
$\eta_{j}(i)=\eta(r_j+i)$ ($i\in I_{\pi, m}$).
Clearly each $\eta_{j}$ is well defined.
For each $\eta_{j}$, by condition~(\ref{cond:independence})
in the induction hypothesis, there exists
\(\tilde{v}^{(j)}\in\mathcal{V}_{m}\) ($1\leq j\leq \xi_m$) such that
\[
 (\tilde{v}^{(j)})_{i}(\tilde{v}^{(j)})_{i+1}\cdots (\tilde{v}^{(j)})_{i+|u_{\pi}|-1}=\eta_{j}(i)
 \quad (i\in I_{\pi,m}).
\]
Put
\[
 \tilde{w}=K_m\Bigg(
 \prod_{a=1}^{\#(\mathscr{Q}_{m})}
 \prod_{\ell=1}^{t_{\pi_a}}c_{\pi_a}^{0}\Bigg)
 \tilde{v}^{(1)}\cdots\tilde{v}^{(\xi_m)}F_m
 \in\mathcal{V}_{m+1}.
\]
Then, for any $i\in I_{\pi,m}$ and any $1\leq j\leq \xi_{m}$,
\[
 \tilde{w}_{r_{j}+i}\tilde{w}_{r_{j}+i+1}\cdots
 \tilde{w}_{r_{j}+i+|u_{\pi}|-1}
 =(\tilde{v}^{(j)})_{i}(\tilde{v}^{(j)})_{i+1}\cdots
 (\tilde{v}^{(j)})_{i+|u_{\pi}|-1}=\eta_{j}(i)
 =\eta(r_{j}+i).
\]
Consequently, by $I_{\pi, m+1}
=\bigcup_{j=1}^{\xi_{m}}(r_{j}+I_{\pi,m})$,
\[
 \tilde{w}_i\tilde{w}_{i+1}\cdots
 \tilde{w}_{i+|u_{\pi}|-1}=\eta(i)
 \quad (i\in I_{\pi,m+1}).
\]

C6-2) If $\pi\in \mathscr{Q}_m\setminus \mathscr{P}_m$,
then \(\pi=\pi_{a}\) for some
\(a\in\{1,\ldots,\#(\mathscr{Q}_{m})\}\).
Let
\[
\zeta_{\pi_{a},\ell}
 =
 \begin{cases}
 0, & \eta(s_{\pi_{a},\ell}+\alpha_{\pi_{a}})=u_{\pi_{a}},\\
 1, & \eta(s_{\pi_{a},\ell}+\alpha_{\pi_{a}})=v_{\pi_{a}},
 \end{cases}
 \quad(1\leq\ell\leq t_{\pi_{a}}),
\]
and
\[
\zeta_{\pi_b,\ell}=0 \quad
(b\in\{1,\ldots,\#(\mathscr{Q}_{m})\}\setminus\{a\},\
1\leq\ell\leq t_{\pi_b}),
\]
and choose
\(\tilde{v}^{(j)}\in\mathcal{V}_{m}\) for
\(1\leq j\leq\xi_m\). Put
\[
 \tilde{w}=K_m
 \Bigg(
 \prod_{b=1}^{\#(\mathscr{Q}_{m})}
 \prod_{\ell=1}^{t_{\pi_b}}
 c_{\pi_b}^{\zeta_{\pi_b,\ell}}
 \Bigg)
 \tilde{v}^{(1)}\cdots\tilde{v}^{(\xi_m)}F_m
 \in\mathcal{V}_{m+1}.
\]
Then, by the defining properties of
\(c_{\pi_a}^{0}\) and \(c_{\pi_a}^{1}\),
\begin{align*}
 &\tilde{w}_{s_{\pi_{a},\ell}+\alpha_{\pi_{a}}}
 \tilde{w}_{s_{\pi_{a},\ell}+\alpha_{\pi_{a}}+1}\cdots
 \tilde{w}_{s_{\pi_{a},\ell}+\alpha_{\pi_{a}}+|u_{\pi_{a}}|-1}\\
 &\quad=
 (c_{\pi_a}^{\zeta_{\pi_a,\ell}})_{\alpha_{\pi_a}}
 (c_{\pi_a}^{\zeta_{\pi_a,\ell}})_{\alpha_{\pi_a}+1}\cdots
 (c_{\pi_a}^{\zeta_{\pi_a,\ell}})_{\alpha_{\pi_a}+|u_{\pi_a}|-1}\\
 &\quad=\eta(s_{\pi_{a},\ell}+\alpha_{\pi_{a}})
 \quad (1 \leq \ell \leq t_{\pi_{a}}).
\end{align*}
Consequently, by
\(I_{\pi_a,m+1}=\{s_{\pi_a,\ell}+\alpha_{\pi_a}:
1\leq\ell\leq t_{\pi_a}\}\),
\[
 \tilde{w}_{i}\tilde{w}_{i+1}\cdots
 \tilde{w}_{i+|u_{\pi_a}|-1}=\eta(i)
 \quad(i\in I_{\pi_a,m+1}).
\]

(C7) For \(m\geq2\) and \(\pi\in\mathscr{P}_{m}\),
\eqref{eq:I-propagation-cardinality} and
\eqref{eq:xi-m-estimate} give
\[
 \frac{\#(I_{\pi,m+1})}{L_{m+1}}
 =\frac{{\xi_{m}}L_{m}}{{L_{m+1}}}
 \frac{\#(I_{\pi,m})}{L_{m}}
 \geq(1-2^{-m-4})\frac{\#(I_{\pi,m})}{L_{m}}.
\]

This completes the induction.
\end{proof}

\begin{corollary}
\label{cor:propagation-density}
Let $\{L_{m}\}$, $\{I_{\pi, m}\}$, and $\{\mathscr{P}_m\}$
be given by Proposition~\ref{prop:block-construction}.
Then, for any \(m\geq2\), any \(\pi\in\mathscr{P}_{m}\),
and any \(n\geq m\),
\[
 \frac{\#(I_{\pi,n})}{L_{n}}
 \geq
 \frac{\#(I_{\pi,m})}{L_{m}}(1-2^{-m-3})>0.
\]
\end{corollary}

\begin{proof}
By
\((1-x)(1-y)=1-x-y+xy\geq 1-x-y\) for \(x, y\in [0,1]\),
induction gives
\[
 \prod_{j=m}^{n-1}(1-2^{-j-4})
 \geq
 1-\sum_{j=m}^{n-1}2^{-j-4}
 \geq
 1-\sum_{j=m}^{\infty}2^{-j-4}
 =1-2^{-m-3}>0.
\]
Combining with condition (\ref{cond:propagation}), we get
\begin{align*}
 \frac{\#(I_{\pi,n})}{L_{n}}
 &\geq (1-2^{-n-3})\frac{\#(I_{\pi,n-1})}{L_{n-1}}
 \geq (1-2^{-n-3})(1-2^{-n-2})
       \frac{\#(I_{\pi,n-2})}{L_{n-2}}\\
 &\geq\cdots\geq
 \frac{\#(I_{\pi,m})}{L_{m}}
 \prod_{j=m}^{n-1}(1-2^{-j-4})
 \geq
 \frac{\#(I_{\pi,m})}{L_{m}} (1-2^{-m-3})>0.
\end{align*}
\end{proof}

We now pass from the finite-word construction to the point
generating the required minimal subshift.

\begin{proposition}
\label{prop:minimal-point}
Let the words \(\{A_{m}\}\) be given by
Proposition~\ref{prop:block-construction}. Then there exists a
unique \(\hat{x}=(\hat{x}_{i})_{i\in\Zp}\in\Sigma_{2}\) such that
\begin{equation}
\label{eq:hat-x-prefixes}
 \hat{x}_{0}\hat{x}_{1}\cdots \hat{x}_{L_{m}-1}
 =A_{m}\quad (\forall m\in\N),
\end{equation}
and
\[
{\hat{x}_{0}=1},
 \quad
{\{n\in\N:\hat{x}_{n}=1\}\subseteq G}.
\]
Moreover, $\hat{x}$ is a minimal point of
$\sigma$, i.e., \(X=\overline{\orb({\hat{x}},\sigma)}\)
is a minimal subshift.
\end{proposition}

\begin{proof}
{Condition (\ref{cond:prefix}) and \(L_{m}\to +\infty\)
give the unique \(\hat{x}\in\Sigma_{2}\) satisfying
\eqref{eq:hat-x-prefixes}.} Condition
(\ref{cond:coordinates}) gives \((A_{m})_{0}=1\) for each
\(m\in \N\), and thus $\hat{x}_{0}=(A_{m})_{0}=1$. {For any
\(n\in\N\), choose \(m\in \N\) such that \(n<L_{m}\).
Then, from condition (\ref{cond:coordinates}),
it follows \({\hat{x}_{n}=(A_{m})_{n}}=(P_{m})_{n},\)
which, together with condition
(\ref{cond:support}), shows that
\(\hat{x}_{n}=1\) implies \(n\in G\).
Therefore}
\({\{n\in\N:\hat{x}_{n}=1\}\subseteq G}.\)

To prove that $\hat{x}$ is a minimal point,
it suffices to check that, for any
\(u=\hat{x}_{0}\hat{x}_{1}\cdots\hat{x}_{n-1}\), \(n\in\N\),
\(N_{\sigma|_{X}}(\hat{x},\Cyl{u}_{X})\)
is syndetic.

Fix any \(n\in\N\). Repeated use of condition
(\ref{cond:concatenation}) yields
\begin{equation}
\label{eq:An-subword-Wn+1}
 A_{n}\in\mathscr{W}_{n}
 \subseteq\bigcap_{w\in\mathscr{W}_{n+1}}\Sub(w)
 \text{ and }
 A_{\ell}\in\mathscr{C}_{\mathscr{W}_{n+1}}(L_{\ell})
 \quad (\forall \ell>n+1),
\end{equation}
implying that, for any \(\ell>n+1\), there exist \(s_{\ell}\in\N\) and
\(w^{(\ell)}_{1},\ldots,w^{(\ell)}_{s_{\ell}}\in\mathscr{W}_{n+1}\)
such that
\(A_{\ell}=w^{(\ell)}_{1}w^{(\ell)}_{2}\cdots w^{(\ell)}_{s_{\ell}}.\)
Moreover, since \(\lvert w^{(\ell)}_{j}\rvert\leq L_{n+1}+1\), we have
\[
 s_{\ell}\geq\frac{L_{\ell}}{L_{n+1}+1}\to +\infty
 \quad
 (\ell\to +\infty).
\]
Then choose a strictly increasing \((\ell_{r})_{r\in \N}\subseteq (n+1, +\infty)\)
such that \(s_{\ell_{r}}\geq r\), and arrange the corresponding decompositions as
\[
\begin{array}{c|ccccc}
 A_{\ell_{1}}&w^{(\ell_{1})}_{1}&\cdots\\
 A_{\ell_{2}}&w^{(\ell_{2})}_{1}&w^{(\ell_{2})}_{2}&\cdots\\
 A_{\ell_{3}}&w^{(\ell_{3})}_{1}&w^{(\ell_{3})}_{2}&w^{(\ell_{3})}_{3}&\cdots\\
 \vdots&\vdots&\vdots&\vdots&\ddots
\end{array}
\]
Since the entries in each column belong to the finite set
\(\mathscr{W}_{n+1}\), there exist subsequences
\((\ell^{(i)}_{r})_{r\in\N}\), \(i\in\N\), of
\((\ell_{r})_{r\in\N}\) satisfying the following conditions:
\begin{itemize}
  \item For any \(i\in\N\),
  \((\ell^{(i+1)}_{r})_{r\in\N}\) is a subsequence of
  \((\ell^{(i)}_{r})_{r\in\N}\);
  \item For any \(i\in\N\), there exists
  \(w_{i}\in\mathscr{W}_{n+1}\) such that
  \(w_{i}^{(\ell^{(i)}_{1})}=w_{i}^{(\ell^{(i)}_{2})}
   =\cdots=w_{i}^{(\ell^{(i)}_{r})}=\cdots=w_{i}.\)
\end{itemize}
In particular, the diagonal sequence \((\ell^{(r)}_{r})_{r\in\N}\) is a subsequence
of \((\ell_{r})_{r\in\N}\) and satisfies
\[
 w^{(\ell^{(r)}_{r})}_{i}=w_{i}
 \quad (1\leq i\leq r),
\]
i.e.,
\[
\begin{array}{c|ccccc}
 A_{\ell^{(1)}_{1}}&w_{1}&\cdots\\
 A_{\ell^{(2)}_{2}}&w_{1}&w_{2}&\cdots\\
 A_{\ell^{(3)}_{3}}&w_{1}&w_{2}&w_{3}&\cdots\\
 \vdots&\vdots&\vdots&\vdots&\ddots
\end{array}
\]
Noting that $A_{\ell^{(r)}_{r}}=\hat{x}_{0}\hat{x}_{1}\cdots
 \hat{x}_{L_{\ell^{(r)}_{r}}-1}$ by \eqref{eq:hat-x-prefixes},
 we obtain
\(\hat{x}=w_{1}w_{2}\cdots.\)
Meanwhile, by \(n\leq L_{n}\), \eqref{eq:hat-x-prefixes} gives
\(u=(A_{n})_{0}(A_{n})_{1}\cdots(A_{n})_{n-1}\).
This, together with $A_{n}\in\bigcap_{w\in\mathscr{W}_{n+1}}\Sub(w)$
(by \eqref{eq:An-subword-Wn+1}) and $w_i\in \mathscr{W}_{n+1}$
($\forall i\in \N$), implies that, for any $i\in \N$, $u\in \Sub(w_i)$.
Since \(|w_{i}|\leq L_{n+1}+1\) for any \(i\in\N\),
\[
 [s,s+2(L_{n+1}+1)]\cap
 N_{\sigma|_{X}}(\hat{x},\Cyl{u}_{X})\neq\varnothing
 \quad (\forall s\in\N).
\]
Thus \(N_{\sigma|_{X}}(\hat{x},\Cyl{u}_{X})\) is syndetic, so
\(\hat{x}\) is minimal.
Consequently \(X\) is a minimal subshift.
\end{proof}

\begin{corollary}
\label{cor:constructed-words-in-language}
For any \(m \in \N\),
\(\mathcal{V}_{m}\subseteq\Lang_{L_{m}}(X).\)
\end{corollary}

\begin{proof}
For any \(w\in\mathcal{V}_{m}\subseteq
\mathscr{W}_{m}\), condition
(\ref{cond:concatenation}), applied to
\(A_{m+1}\in\mathscr{W}_{m+1}\), gives
\(w\in\Sub(A_{m+1})\). This, together with
\eqref{eq:hat-x-prefixes}, yields
\(\mathcal{V}_{m}\subseteq \Lang_{L_{m}}(X)\).
\end{proof}

\begin{proposition}
\label{prop:upe}
The TDS \((X,\sigma|_{X})\) in Proposition
\ref{prop:minimal-point} has u.p.e.
\end{proposition}

\begin{proof}
For any \((p,q)\in X\times X\setminus \{(x, x): x\in X\}\),
and any disjoint closed neighborhoods \(U\) and \(V\) of
\(p\) and \(q\), there exist two distinct words \(u\), \(v\)
with $|u|=|v|\geq 1$ such that \(p\in\Cyl{u}_{X}\subseteq U\) and
\(q\in\Cyl{v}_{X}\subseteq V.\) By
$X=\overline{\orb({\hat{x}},\sigma)}$ and \eqref{eq:hat-x-prefixes},
there exists an \(r\in \N\)
such that \(u, v\in \Sub(A_{r})\).
Choose \(m_{1}=\max\{r,|u|\}+1\).
From condition (\ref{cond:prefix}), it follows
\((A_{m_{1}-1})_{0}(A_{m_{1}-1})_{1}\cdots(A_{m_{1}-1})_{L_{r}-1}=A_{r},\)
so condition (\ref{cond:pairs}) gives
\({\{u,v\}\in\mathscr{P}_{m_{1}}}.\)
Clearly \(\#(I_{\{u,v\},n})\geq 1\) for each \(n\geq m_{1}\) by
\eqref{eq:P-m-and-Ipim}.

Consider the open cover
\(\mathcal{U}_{u,v}=\{X\setminus\Cyl{u}_{X}, X\setminus\Cyl{v}_{X}\}.\)
Given any fixed \(n \geq m_{1}\),
for each map \({\eta}\colon I_{{\{u,v\}},n}\to\{u,v\},\)
by condition (\ref{cond:independence}), there exists
\(w_{{\eta}}\in\mathcal{V}_{n}\) satisfying
\begin{equation}
\label{eq:w-eta=eta}
 (w_{{\eta}})_{i}
 (w_{{\eta}})_{i+1}\cdots
 (w_{{\eta}})_{i+|u|-1}
 ={\eta}(i)
 \quad (i\in I_{\{u, v\}, n}).
\end{equation}
By Corollary~\ref{cor:constructed-words-in-language} and
\(\sigma(X)\subseteq X\),
there exists
\(z_{{\eta}}\in X\) such that
\begin{equation}
\label{eq:z-eta=w-eta}
 (z_{{\eta}})_{0}(z_{{\eta}})_{1}\cdots
 (z_{{\eta}})_{L_{n}-1}=w_{{\eta}}.
\end{equation}
For any \(\eta, \eta': I_{{\{u,v\}},n}\to\{u,v\}\)
with \({\eta}\neq {\eta}'\), there exists some
\(i\in I_{{\{u,v\}},n}\) such that
\((w_{{\eta}})_{i}(w_{{\eta}})_{i+1}\)
 \(\cdots(w_{{\eta}})_{i+|u|-1}
 ={\eta}(i)\neq {\eta}'(i)
 =(w_{{\eta}'})_{i}
 (w_{{\eta}'})_{i+1}\cdots
 (w_{{\eta}'})_{i+|u|-1},\)
implying $w_{\eta}\neq w_{\eta'}$,
and thus $z_{\eta}\neq z_{\eta'}$.
Consequently,
\[
 \#(\{z_{{\eta}}:
 \eta\in \{u, v\}^{I_{\{u, v\},n}}\})
 =\#(\{u, v\}^{I_{\{u, v\},n}})
 =2^{\#(I_{\{u, v\}, n})}.
\]

On the other hand, by \eqref{eq:w-eta=eta} and
\eqref{eq:z-eta=w-eta},
we have that, for any \(i\in I_{{\{u,v\}},n}\),
\[
 \begin{cases}
 z_{{\eta}}\in\sigma^{-i}(X\setminus\Cyl{u}_{X})
 \Longleftrightarrow \sigma^{i}(z_{{\eta}})\in X\setminus \Cyl{u}_{X}
 \Longleftrightarrow \sigma^{i}(z_{{\eta}})\in \Cyl{v}_{X}
 \Longleftrightarrow \eta(i)=v,\\
 z_{{\eta}}\in\sigma^{-i}(X\setminus\Cyl{v}_{X})
 \Longleftrightarrow \sigma^{i}(z_{{\eta}})\in X\setminus \Cyl{v}_{X}
 \Longleftrightarrow \sigma^{i}(z_{{\eta}})\in \Cyl{u}_{X}
 \Longleftrightarrow \eta(i)=u,
 \end{cases}
\]
implying
\(\#(A\cap \{z_{{\eta}}:
 \eta\in \{u,v\}^{I_{\{u,v\},n}}\}) \leq 1\)
 (\(\forall A\in\bigvee_{i\in I_{\{u, v\}, n}}
 \sigma^{-i}\mathcal{U}_{u,v}\)),
and thus
\[
 N\Bigg(
 \bigvee_{i\in I_{{\{u,v\}},n}}\sigma^{-i}\mathcal{U}_{u,v}
 \Bigg) \geq \#(\{z_{{\eta}}:
 \eta\in \{u,v\}^{I_{\{u,v\},n}}\})
 = 2^{\#(I_{{\{u,v\}},n})}.
\]
Since
\(\bigvee_{i=0}^{L_{n}-1}\sigma^{-i}\mathcal{U}_{u,v}\) refines
{\(\bigvee_{i\in I_{\{u,v\},n}}
\sigma^{-i}\mathcal{U}_{u,v}\)} by
$I_{\{u,v\}, n}\subseteq
 \{0, \ldots, L_{n}-1\}$,
this, together with
Corollary~\ref{cor:propagation-density},
implies
\begin{equation}
\label{eq:cover-lower-bound}
 N\left(
 \bigvee_{i=0}^{L_{n}-1}\sigma^{-i}\mathcal{U}_{u,v}
 \right)\geq 2^{\#(I_{\{u, v\}, n})}
\geq
 2^{\frac{\#(I_{\{u, v\}, m_{1}})}{L_{m_{1}}}
 (1-2^{-{m_{1}}-3}) L_{n}}
 \geq 2^{\frac{(1-2^{-{m_{1}}-3})L_{n}}{L_{m_{1}}}}
 \quad (\forall n \geq m_{1}).
\end{equation}
Again, by the fact that \(\{X\setminus U, X\setminus V\}\)
refines \(\mathcal{U}_{u,v}\), we get
\begin{equation*}
%%\label{eq:positive-cover-entropy}
\begin{aligned}
 & h_{\mathrm{top}}(\sigma|_{X},\{X\setminus U, X\setminus V\})
 \geq h_{\mathrm{top}}(\sigma|_{X}, \mathcal{U}_{u,v})
 =
 \lim_{n\to\infty}\frac{1}{n}
 \log {N}\left(
 \bigvee_{i=0}^{n-1}\sigma^{-i}\mathcal{U}_{u,v}
 \right)\\
 & \quad =
 \lim_{n\to\infty}\frac{1}{L_{n}}
 \log {N}\left(
 \bigvee_{i=0}^{L_{n}-1}\sigma^{-i}\mathcal{U}_{u,v}
 \right)\geq
 \frac{1-2^{-{m_{1}}-3}}{L_{{m_{1}}}}
 \log 2> 0.
\end{aligned}
\end{equation*}

Thus \((p,q)\) is an entropy pair, and hence
\((X,\sigma|_{X})\) has u.p.e.
\end{proof}

\begin{theorem}
\label{thm:upe-system}
There exists \({\hat{x}}\in \Sigma_{2}\) satisfying
the following conditions:
\begin{enumerate}
\renewcommand{\labelenumi}{\textup{(\roman{enumi})}}
\renewcommand{\theenumi}{\roman{enumi}}
\item\label{item:upe-support} \({{\hat{x}_{0}}=1}\) and
\({\{n\in\N: {\hat{x}_{n}}=1\}\subseteq G}\);
\item\label{item:upe-minimal} The TDS \((X,\sigma|_{X})\) is minimal and has u.p.e.,
where \(X=\overline{\orb({\hat{x}},\sigma)}\).
\end{enumerate}
\end{theorem}

\begin{proof}
It follows directly from Propositions~\ref{prop:block-construction},
\ref{prop:minimal-point}, and \ref{prop:upe}.
\end{proof}

\begin{remark}
It follows from \cite[Theorem~3.2~\textup{(1)}]{GlasnerYeLocal} and
\cite[Theorem~3.6]{HuangShaoYeProximal} that every minimal TDS with
u.p.e.\ has no distal points. In particular, the system
\((X, \sigma|_{X})\) has no distal points.
\end{remark}

We conclude by relating this minimal u.p.e.\ subshift to
the zero-entropy subshift constructed in
Subsection~\ref{sec:zero-entropy}.

\begin{theorem}
\label{thm:X-not-perp-Y}
\((X,\sigma|_{X})\not\perp(Y,\sigma|_{Y}).\)
\end{theorem}

\begin{proof}
Let
\(J=\omega_{\sigma\times\sigma}((\hat{x},a)).\)
Similar to the argument in the proof of
Proposition~\ref{prop:product-recurrence}, it is easy to see that
\(J\) is a joining of \((X,\sigma|_{X})\) and \((Y,\sigma|_{Y})\).
By Theorem~\ref{thm:Y-system}~\textup{(\ref{item:Y-return-set})}
and Theorem~\ref{thm:upe-system}~\textup{(\ref{item:upe-support})},
we have
\[
 \{(\sigma\times\sigma)^{n}({\hat{x}},a)
 : n\in \N\}\subseteq (X\times Y)\setminus
 (\Cyl{1}_{X}\times\Cyl{1}_{Y}).
\]
Noting that \(\Cyl{1}_{X}\times\Cyl{1}_{Y}\)
is a nonempty clopen subset of \(X\times Y\),
we get
\[
J=\omega_{\sigma\times\sigma}((\hat{x},a))
\subseteq \overline{\{(\sigma\times\sigma)^{n}({\hat{x}},a)
 : n\in \N\}}\subseteq (X\times Y)\setminus
 (\Cyl{1}_{X}\times\Cyl{1}_{Y})\subsetneq X\times Y,
\] so \(J\neq X\times Y\). Hence
\((X,\sigma|_{X})\not\perp(Y,\sigma|_{Y}).\)
\end{proof}

\subsection[Proofs of Theorems~\ref{thm:disjointness} and \ref{thm:main}]
{Proofs of Theorems~\ref{thm:disjointness} and \ref{thm:main}}
\label{Sec:4}

\begin{proof}[Proof of Theorem~\ref{thm:disjointness}]
Let \(X\), \(\hat{x}\), \(Y\), and \(a\) be given by
Theorems~\ref{thm:Y-system} and \ref{thm:upe-system}.
By Theorem~\ref{thm:upe-system} \textup{(\ref{item:upe-minimal})},
\((X,\sigma|_{X})\) is minimal and has u.p.e., and hence is a minimal
diagonal TDS. This, together with Lemma~\ref{thm:HPY}, implies
\((X,\sigma|_{X})\in\Mzero^{\perp}.\)

Meanwhile, Theorem~\ref{thm:Y-system}~\textup{(\ref{item:Y-class})}
gives \((Y,\sigma|_{Y})\in\Ezero\). Together with
Theorem~\ref{thm:X-not-perp-Y}, this implies
\((X,\sigma|_{X})\notin \Ezero^{\perp}\).
Therefore $(X, \sigma|_{X})\in \Mzero^{\perp}
\setminus \Ezero^{\perp}$.
\end{proof}

\begin{proof}[Proof of Theorem~\ref{thm:main}]
Let \({\hat{x}}\) and
\(X=\overline{\orb({\hat{x}},\sigma)}\)
be given by Theorem~\ref{thm:upe-system}. Since
\((X,\sigma|_{X})\) is minimal and has u.p.e., it is a
minimal diagonal TDS.  This, together with
Proposition~\ref{prop:product-recurrence}, yields
${\hat{x}} \in \Fps \PRzero.$

Let \(a\in Y\) be the point defined in Subsection~\ref{sec:zero-entropy}.
By Lemmas~\ref{lem:zero-entropy} and
\ref{lem:a-recurrent}, \(a\) is an \(\Fpubd\)-recurrent point in the
zero-entropy TDS \((Y,\sigma|_{Y})\). For the open neighborhood
\(\Cyl{1}_{X}\times \Cyl{1}_{Y}\) of $({\hat{x}}, a)$,
Theorem~\ref{thm:Y-system}~\textup{(\ref{item:Y-return-set})}
and Theorem~\ref{thm:upe-system}~\textup{(\ref{item:upe-support})} give
\[
 N_{(\sigma|_{X})\times(\sigma|_{Y})}(({\hat{x}},a),
 \Cyl{1}_{X}\times \Cyl{1}_{Y})
 =N_{\sigma|_{X}}({\hat{x}},\Cyl{1}_{X})
 \cap N_{\sigma|_{Y}}(a, \Cyl{1}_{Y})
 \subseteq (G\cap \N)\cap (H\cap \N)
 =\varnothing,
\]
implying \(({\hat{x}},a)\notin
\mathrm{Rec}(X\times Y,(\sigma|_{X})\times(\sigma|_{Y}))\),
and thus \({\hat{x}}\notin\Fpubd\PRzero\).

Therefore
\({\hat{x}}\in\Fps\PRzero\setminus
\Fpubd\PRzero.\)
\end{proof}

\end{document}